\documentclass{amsart}

\usepackage{amsmath}
\usepackage{amsthm} 
\usepackage{graphicx}
\usepackage{color}
\usepackage{scalerel}

\usepackage{enumerate}
\usepackage{amsfonts}
\usepackage{amssymb}
\usepackage[hidelinks]{hyperref}
\usepackage[capitalize]{cleveref}

\usepackage{nicefrac}
\usepackage{xfrac}

\usepackage{tipa}
\usepackage[normalem]{ulem}

\usepackage{mathtools}

\newcommand{\e}{\varepsilon}
\newcommand{\tp}{\mathrm{tp}}

\newcommand{\Lc}{\mathcal{L}}
\newcommand{\Uc}{\mathcal{U}}

\newcommand{\Ac}{\mathcal{A}}
\newcommand{\Dc}{\mathcal{D}}
\newcommand{\Cc}{\mathcal{C}}
\newcommand{\Bc}{\mathcal{B}}
\newcommand{\Rb}{\mathbb{R}}
\newcommand{\Pb}{\mathbb{P}}
\newcommand{\Eb}{\mathbb{E}}
\newcommand{\frk}{\mathfrak}

\newcommand{\bbar}{\bar{b}}
\newcommand{\ybar}{\bar{y}}
\newcommand{\xbar}{\bar{x}}
\newcommand{\zbar}{\bar{z}}
\newcommand{\fbar}{\bar{f}}

\newcommand{\hbarr}{\bar{h}}
\newcommand{\pn}{\#}

\def\Av{\operatorname{Av}}
\def\Avtp{\operatorname{Avtp}}
\def\tp{\operatorname{tp}}
\def\KM{\operatorname{KM}}

\providecommand{\dotdiv}{
  \mathbin{
    \vphantom{+}
    \text{
      \mathsurround=0pt 
      \ooalign{
        \noalign{\kern-.35ex}
        \hidewidth$\smash{\cdot}$\hidewidth\cr 
        \noalign{\kern.35ex}
        $-$\cr 
      }%
    }%
  }%
}

\newcommand{\cB}{\mathcal{B}}
\newcommand{\cF}{\mathcal{F}}
\newcommand{\cU}{\mathcal{U}}
\newcommand{\cL}{\mathcal{L}}
\newcommand{\cS}{\mathcal{S}}
\newcommand{\R}{\mathbb{R}}
\newcommand{\seq}{\subseteq}
\newcommand{\kM}{\mathfrak{M}}
\newcommand{\abar}{\bar{a}}
\newcommand{\cbar}{\bar{c}}
\newcommand{\gbar}{\bar{g}}

\newcommand{\fsd}{\mathbin{\raisebox{.25pt}{\textnormal{\footnotesize{$\triangle$}}}}}
\newcommand{\ssd}{\bigtriangleup}

\newcommand{\subt}{\hspace{-1pt}\scaleto{\triangle}{5pt}\hspace{-1pt}}
\newcommand{\ssubt}{\hspace{-1pt}\scaleto{\triangle}{3pt}\hspace{-1pt}}

\newtheorem{thm}{Theorem}[section]
\newtheorem{theorem}[thm]{Theorem}
\newtheorem{prop}[thm]{Proposition}
\newtheorem{proposition}[thm]{Proposition}
\newtheorem{lem}[thm]{Lemma}
\newtheorem{lemma}[thm]{Lemma}
\newtheorem{cor}[thm]{Corollary}
\newtheorem{corollary}[thm]{Corollary}
\newtheorem{fact}[thm]{Fact}

\newtheorem{question}[thm]{Question}
\theoremstyle{definition}
\newtheorem{defn}[thm]{Definition}
\newtheorem{definition}[thm]{Definition}
 
\newtheorem{remark}[thm]{Remark}

\def\Ind{\setbox0=\hbox{$x$}\kern\wd0\hbox to 0pt{\hss$\mid$\hss}
\lower.9\ht0\hbox to 0pt{\hss$\smile$\hss}\kern\wd0}

\def\Notind{\setbox0=\hbox{$x$}\kern\wd0\hbox to 0pt{\mathchardef
\nn=12854\hss$\nn$\kern1.4\wd0\hss}\hbox to
0pt{\hss$\mid$\hss}\lower.9\ht0 \hbox to 0pt{\hss$\smile$\hss}\kern\wd0}

\newcommand{\sqin}{%
  \mathrel{\vphantom{\sqsubset}\text{%
      \mathsurround=0pt
      \ooalign{$\sqsubset$\cr$-$\cr}%
    }}%
}

\newcommand{\miff}{\makebox[.5in]{$\Leftrightarrow$}}
\newcommand{\inv}{^{\text{-}1}}

\newcommand{\fim}{\textit{fim}}
\newcommand{\famm}{\textit{fam}}
\newcommand{\dfs}{\textit{dfs}}

\newcommand{\Toot}{\Leftrightarrow}

\newcommand{\Ti}{T^\infty_{\textnormal{\sfrac{1}{2}}}}

\renewcommand{\epsilon}{\varepsilon}

\allowdisplaybreaks 

\begin{document}

\title{Generic stability, randomizations, and NIP formulas}

\date{August 3, 2023}
\author[G. Conant]{Gabriel Conant}

\address{Department of Mathematics\\
The Ohio State University\\
Columbus, OH, 43210, USA}
\email{conant.38@osu.edu}%

\author[K. Gannon]{Kyle Gannon}

\address{Beijing International Center for Mathematical Research (BICMR)\\
Peking University\\
No. 5 Yiheyuan Road, Haidian District, Beijing, China}
\email{kgannon@bicmr.pku.edu.cn}

\author[J. Hanson]{James E. Hanson}

\address{Department of Mathematics\\
University of Maryland,\\ College Park, MD, 20742, USA}
\email{jhanson9@umd.edu}

\thanks{GC was partially supported by NSF grant DMS-2204787.}

\maketitle

\begin{abstract}

We prove a number of results relating the concepts of Keisler measures, generic stability, randomizations, and NIP formulas. Among other things, we do the following: 

\begin{enumerate}
    \item We introduce the notion of a \emph{Keisler-Morley measure}, which plays the role of a Morley sequence for a Keisler measure. We prove that if $\mu$ is \fim\ over $M$, then for any Keisler-Morley measure $\lambda$ in $\mu$ over $M$ and any formula $\varphi(x,b)$, $\lim_{i \to \infty} \lambda(\varphi(x_i,b)) = \mu(\varphi(x,b))$. We also show that any measure satisfying this conclusion must be \famm.
    \item We study the map, defined by Ben Yaacov, taking a definable measure $\mu$ to a type $r_\mu$ in the randomization. 
    We prove that this map commutes with Morley products, and  that if $\mu$ is \fim\ then $r_\mu$ is generically stable.
    \item  We characterize when generically stable types are closed under Morley products by means of a variation of \emph{ict}-patterns. Moreover, we show that NTP$_2$ theories satisfy this property.
    \item We prove that if a local measure admits a suitably tame global extension, then it has finite packing numbers with respect to any definable family. We also characterize NIP formulas via the existence of tame extensions for local measures. 
\end{enumerate}
    
\end{abstract}

 \section*{Introduction}

 This paper continues a line of research started in \cite{CoGa} and continued in \cite{CoGaHa}. The general goal is to study the behavior of Keisler measures without a global NIP assumption on the ambient theory. One of the key contributions of \cite{CoGaHa} was  exhibiting the failure of certain properties of Keisler measures outside of NIP (e.g., associativity and preservation of Borel definability for Morley products). Here  we focus on obtaining positive results. 

 In NIP theories, a global type is called \emph{generically stable} if it is definable and finitely satisfiable in some small model (see \cite{HP}). Since the definitions of definability and finite satisfiability transfer easily to Keisler measures, one obtains a reasonable notion of generic stability for  measures in the NIP setting, which has a number of useful and striking characterizations  (see \cite{HPS}).
 
 Without assuming NIP, the situation is different. The notion of generic stability for global types in arbitrary theories\footnote{By a type (or measure) ``in" a theory $T$, we mean a type (or measure) over a subset of some model of $T$. We note that the second author strenuously objects to this sloppy terminology.}  was initially clarified in work of Pillay and Tanovi\v{c} \cite{PiTa}, and further studied by Garc\'{i}a, Onshuus, and Usvyatsov \cite{GaOnUs}. Moreover, this notion is known to \emph{not} be the same as ``definable and finitely satisfiable" in general \cite{CoGa}.
A motivating question for the first part of this paper is the following.
\begin{quotation}
    What is the correct definition of ``generic stability" for a Keisler measure in an arbitrary theory?
\end{quotation}
Previous work has supported the class of frequency interpretation measures (\fim) as a potential answer. In particular, the first two authors show in \cite{CoGa} that \fim\ characterizes generic stability for types. Moreover, in \cite{CoGaHa}, we prove that  \fim\ measures commute with  Borel-definable measures (or more generally, with any invariant measure admitting well-defined Morley products). For NIP theories, this latter condition is one of the characterizations of generic stability for Keisler measures from \cite{HPS}. Despite these results from \cite{CoGa} and \cite{CoGaHa}, it remains an open problem to prove that \fim\ can be characterized via a definition that ``clearly" captures the essence of  generic stability as defined for types.  Therefore, in the first part of the present paper, we will approach this problem from a different direction. In particular, we  identify other potential notions of ``generic stability" for a Keisler measure, and prove implications between them.

Consider a first-order theory $T$ and a Keisler measure $\mu$ over a monster model $\cU\models T$. For purposes of discussion, let us assume $\mu$ is definable over some small model $M\prec\cU$ (see Definition \ref{def:basic-measure}). We will work with three properties of $\mu$ resembling generic stability  of global types  (Fact \ref{fact:genstab}).
\begin{enumerate}[$(1)$]
    \item $\mu$ is a \emph{frequency interpretation measure} (\fim). 
    
    This notion was defined by Hrushovski, Pillay, and Simon \cite{HPS}, and essentially means that $\mu$ can be locally finitely approximated by average measures over a randomly chosen finite tuple (in the sense of the Morley powers $\mu^{(n)}$). See Definition \ref{def:fimDL} for details.

    \item The ``randomization" $r_\mu$ is generically stable as a type in the continuous theory $T^R$ (the randomization of $T$). 

    Here we use a process (due to Ben Yaacov \cite{BYT}) for transferring a definable Keisler measure $\mu$ in a theory $T$ to a definable global type $r_\mu$ in $T^R$. Roughly speaking, $r_\mu$ is the ``generic" linear extension of $\mu$ viewed as a type in $T^R$ in a suitable way.  In Section \ref{sec:ran-def-meas}, we clarify the details of this construction and  show that the map  $\mu\mapsto r^\mu$ commutes with Morley products (Proposition \ref{prop:prod-same-1}).

      \item $\mu$ is \emph{self-averaging}.

    This concept is introduced in Definition \ref{def:GSM}, and captures the idea that the behavior of $\mu$ is controlled by ``Morley sequences in $\mu$" by which we mean global measures extending $\mu^{(\omega)}|_M$ (see Definition \ref{defn:fim-measure}). 
\end{enumerate}

In Section \ref{sec:self-averaging} we establish  properties of self-averaging measures (some of which are discussed further below). Then in Section \ref{sec:randomization}, we  prove the following implications:\medskip
\begin{center}
     $\mu$ is \fim \makebox[.3in]{$\Rightarrow$}  $r_\mu$ is generically stable \makebox[.3in]{$\Rightarrow$}  $\mu$ is self-averaging. 
\end{center}\medskip
These implications are obtained from Proposition \ref{prop:rpgs}, which is  the main result of Section \ref{sec:randomization}.  In particular, given $\mu$, we obtain two sequences $(\psi^1_n(x_1,\ldots,x_n))_{n=1}^\infty$ and $(\psi^2_n(x_1,\ldots,x_n))_{n=1}^\infty$ of continuous formulas in $T^R$ such that $\psi^2_n(\xbar)\leq\psi^1_n(\xbar)$ (pointwise) for all $n$; and in Proposition \ref{prop:rpgs}, we show that for $t\in\{1,2\}$, condition $(t)$ holds if and only if $\lim_{n\to\infty}\psi^t_n(\xbar)^{r_\mu^{(n)}}=0$. This immediately yields $(1)\Rightarrow (2)$ (see Corollary \ref{cor:fim-transfer}), and leads to a short proof of $(2)\Rightarrow (3)$ (see Remark \ref{rem:sa-random}). We do not currently know whether either of the above implications are strict. However, it is worth noting that the formulas $\psi^1_n(\xbar)$ and $\psi^2_n(\xbar)$ from the statement of Proposition \ref{prop:rpgs} are \emph{nearly} identical, and the weak inequality between the formulas is simply an instance of Jensen's inequality. We also point out that for a global \emph{type}, (1), (2), and (3) are equivalent and coincide with the standard notion of generic stability  (see Fact \ref{fact:CG}, Corollary \ref{cor:fimtypes}, and Proposition \ref{prop:rpgs}). Moreover,   (1), (2), and (3) are equivalent for   measures  in NIP theories (see Corollary \ref{cor:saNIP} and Remark \ref{rem:fim-transfer-NIP}). Finally, in Section \ref{sec:fam} we show that \famm\ measures transfer to \famm\ types, and in Section \ref{sec:dfs} we give an example which demonstrates \emph{dfs} measures need not transfer to \emph{dfs} types. 

For some readers, the involvement of randomizations might appear unmotivated. An informative summary of this perspective can be found in the introduction of \cite{HPS}, where the main characterizations of generically stable measures in NIP theories are proved. The discussion of randomizations in \cite{HPS} is partly in order to explain the proof of \cite[Lemma 2.10]{HPS}, which has aspects bearing  some resemblance to the notion of self-averaging. That being said, in Theorem \ref{thm:fim-main} we will give direct proofs of the following implications (which do not require any familiarity with randomizations):\medskip
\begin{center}
     $\mu$ is \fim \makebox[.3in]{$\Rightarrow$}   $\mu$ is self-averaging \makebox[.3in]{$\Rightarrow$} $\mu$ is \famm.
\end{center}\medskip
Recall that \famm\ (``finitely approximated measure") is a weakening of \fim\ in which one only asserts the existence of finite approximations  (see Definition \ref{def:fam}). So even though we  do not know whether (1), (2), and (3) above are equivalent, the above implications show that all three properties sit in the same place among other commonly studied notions. On the other hand, the implication from self-averaging to \famm\ is known to be strict, due to the existence of types that are \famm\ but not generically stable (examples can be found in \cite[Theorem 5.8]{CoGa} and \cite[Section 8]{CoGaHa}).

None of the properties (1), (2), or (3) above is currently known to be preserved by Morley products. Indeed, the question of whether generic stability for \emph{types} is preserved by Morley products remains an unsettling open question. In Section \ref{sec:NTP2}, we provide some partial progress. First, we give a combinatorial characterization of theories in which generically stable types are closed under Morley products. This characterization is formulated using a ternary variation of ict-patterns involving stable sequences (see Definition \ref{def:gss}). We then show that such a pattern cannot exist in an NTP$_2$ theory. Therefore, generically stable types in NTP$_2$ theories are closed under Morley products. This generalizes the NIP case, which is an easy consequence of the characterization of generic stability in \cite[Proposition 3.2]{HP}. 

Finally, in Section \ref{sec:local}, we contribute to the ongoing development of ``local NIP", i.e., the  study of NIP formulas in arbitrary theories. Given a formula $\varphi(x,y)$ and a  $\varphi$-measure $\mu$ over a model $M\prec\cU$, we show that if $\mu$ admits a global extension satisfying a strong definability condition (implied for example by ``local \dfs"), then for any $\epsilon>0$ there is finite bound on the size of an $\epsilon$-separated family of instances of $\varphi(x,y)$ (see Proposition \ref{prop:phipack}). We then obtain a characterization of NIP formulas via the existence of suitably tame (e.g., \dfs, smooth, etc.) extensions for local Keisler measures over models (see Theorem \ref{thm:localNIPchar}).

{\subsection*{Acknowledgements} We thank Will Johnson for observing Remark \ref{rem:fim-transfer-NIP}, and Karim Khanaki for  comments on an earlier draft.

\section{Preliminaries}

 Given $r,s\in \R$ and $\epsilon > 0$, we write $r \approx_{\epsilon} s$ to mean $|r - s| <\epsilon$. We frequently use the notation $a_{<\omega}$ for an $\omega$-indexed sequence $(a_i)_{i<\omega}$ of elements from some set. Similarly, $a_{<n}$ denotes a finite sequence $(a_i)_{i<n}$.

  Let $T$ be a complete first-order theory with monster model $\cU$. We write $A\subset\cU$ to denote that $A$ is a small subset of $\cU$, and we write $M\prec\cU$ to denote that $M$ is a small elementary substructure of $\cU$. Throughout the paper, the letters $x$, $y$, and $z$ represent tuples of variables. However, we also write $\xbar$, $\ybar$, and $\zbar$ when we wish to emphasize the fact that we are working with tuples. 

\subsection{Generically stable types}\label{sec:gstypes-def}
The notion of a generically stable global type has several different characterizations, which we now recall.

\begin{fact} \label{fact:genstab}
Given $p \in S_{x}(\mathcal{U})$ and $M\prec\cU$, the following are equivalent.
\begin{enumerate}[$(i)$]
    \item For any Morley sequence $(a_i)_{i < \omega}$ in $p$ over $M$, $\lim_{i \to \infty} \tp(a_i/\mathcal{U}) = p$. 
    \item For any $\mathcal{L}$-formula $\varphi(x,y)$, there exists some $n_\varphi\geq 1$ such that for any $b\in \cU^y$ and 
     any Morley sequence $(a_i)_{i < \omega}$ in $p$ over $M$, 
    \[
    \varphi(x,b)\in p\miff |\{i<\omega:\cU\models\varphi(a_i,b)\}|\geq n_\varphi.
    \]
      
    \item There does not exist an $\mathcal{L}$-formula $\varphi(x,y)$, a Morley sequence $(a_i)_{i<\omega}$ in $p$ over $M$, and a sequence $(b_i)_{i<\omega}$ from $\cU^y$ such that $\cU\models\varphi(a_i,b_j)$ if and only if $i\leq j$. 
\end{enumerate}
\end{fact}

For proofs of the previous characterizations, the reader may consult \cite{PiTa} or  \cite[Section 3]{CoGa}. We say that $p\in S_x(\cU)$ is \textbf{generically stable over $M$} if the previous equivalent conditions hold. Generically stable types can also be understood via ``stable sequences", which we now recall.

\begin{definition}\label{def:stabseq}
    A sequence $(a_i)_{i<\omega}$ in $\mathcal{U}^{x}$ is  \textbf{stable} if  for any $\Lc$-formula $\varphi(x,y)$ there is some natural number $n$ such that for any $b\in\Uc^y$, either 
    \[
    |\{i<\omega:\cU\models \varphi(a_i,b)\}|\leq n\makebox[.5in]{or}|\{i<\omega:\cU\models \neg\varphi(a_i,b)\}|\leq n.
    \]
\end{definition}

Any stable sequence $a_{<\omega}$ determines a unique global type in $S_{x}(\mathcal{U})$, typically called the \textit{average type} of the sequence. More explicitly, we let $\Avtp(a_{<\omega})$ be the type in $S_{x}(\mathcal{U})$ such that an $\mathcal{L}_{\cU}$-formula $\varphi(x)$ is in $\Avtp(a_{<\omega})$ if and only if $\{i<\omega:\cU\models\varphi(a_i)\}$ is infinite. The next fact is also well known.

\begin{fact}\label{fact:gen-stab-seq-and-types}$~$
\begin{enumerate}[$(a)$]
    \item Suppose that $a_{<\omega}$ is a sequence in $\mathcal{U}^{x}$. If $a_{< \omega}$ is stable, then the global type $\Avtp(a_{<\omega})$ is $\{a_i: i < \omega\}$-invariant and generically stable. 
    \item An $M$-invariant type $p\in S_x(\cU)$ is generically stable if and only if $p=\Avtp(a_{<\omega})$ for any Morley sequence $a_{<\omega}$ in $p$ over $M$.
  \end{enumerate}
\end{fact}

\subsection{Keisler measures} \label{sec:gloss}
Here we give a terse review of definitions and facts involving Keisler measures. We refer the reader to \cite{CoGaHa} for further details and discussion, including original sources (e.g., \cite{HPP,HPS}).

Given a tuple of variables $x$, let $\kM_x(\cU)$ denote the space of Keisler measures on $\cL_{\cU}$-formulas in free variables $x$. As usual, we identify an element of $\kM_x(\cU)$ with a regular Borel probability measure on $S_x(\cU)$. Given $a\in\cU^x$, we let $\delta_a\in\kM_x(\cU)$ denote the Dirac measure concentrating on the realized type $\tp(a/\cU)$.

\begin{definition}[invariant/definable/finitely satisfiable]\label{def:basic-measure} Fix $\mu \in \mathfrak{M}_{x}(\mathcal{U})$. 
\begin{enumerate}[$(1)$]
    \item We say that $\mu$ is \textbf{invariant} if there is a  set $A \subset \mathcal{U}$ such that for any $\mathcal{L}$-formula $\varphi(x,y)$ and $b,c \in \mathcal{U}^{y}$, if $b\equiv_A c$ then $\mu(\varphi(x,b)) = \mu(\varphi(x,c))$. In this case, we also say that $\mu$ is \textbf{$A$-invariant}, and we define the map $F^\varphi_{\mu,A}\colon S_y(A)\to [0,1]$ via $F^\varphi_\mu(q)=\mu(\varphi(x,b))$ for some/any $b\models q$.
    \item We say that $\mu$ is \textbf{Borel-definable} if there is a set $A\subset\cU$ such that $\mu$ is $A$-invariant and, for every $\mathcal{L}$-formula $\varphi(x,y)$, the map $F_{\mu,A}^{\varphi}$ is Borel. In this case, we also say that $\mu$ is \textbf{Borel-definable over $A$.} 
    \item We say that $\mu$ is \textbf{definable} if there is a set $A\subset\cU$ such that $\mu$ is $A$-invariant and, for every $\mathcal{L}$-formula $\varphi(x,y)$, the map $F_{\mu,A}^{\varphi}$ is continuous. In this case, we also say that $\mu$ is \textbf{$A$-definable} or \textbf{definable over $A$}. 
    \item We say that $\mu$ is \textbf{finitely satisfiable in $A\subset\cU^x$} if for any $\cL_{\cU}$-formula $\varphi(x)$ with $\mu(\varphi(x))>0$, there is some $a\in A^x$ such that $\cU\models\varphi(a)$. 
    \item We say that $\mu$ is \textbf{\textit{dfs}} if it is definable and finitely satisfiable in some $A\subset\cU$. In this case, we also also say that $\mu$ is \textbf{\dfs\ over $A$}.
    \end{enumerate}
\end{definition}

\begin{definition}[average measures] Given a tuple $\abar = (a_1,\ldots,a_n)$ of elements of $\mathcal{U}^{x}$, define $\Av(\abar)=\frac{1}{n}\sum_{i\leq n}\delta_{a_i}$. Note $\Av(\abar)\in \mathfrak{M}_{x}(\mathcal{U})$.
\end{definition}

\begin{definition}[\famm]\label{def:fam}
    We say that $\mu\in\kM_x(\cU)$ is a \textbf{finitely approximated measure} (or \textbf{\textit{fam}}) if there is some $A\subset\cU^x$ such that, for any $\cL$-formula $\varphi(x,y)$ and any $\epsilon>0$, there is some $\abar\in (A^x)^{<\omega}$ such that for any $b\in\cU^y$, $\mu(\varphi(x,b))\approx_\epsilon\Av(\abar)(\varphi(x,b))$. In this case we also say that $\mu$ is \textbf{\famm\ over $A$}.
\end{definition}

\begin{definition}[Morley product] Fix $\mu \in \mathfrak{M}_{x}(\mathcal{U})$ and $\nu \in \mathfrak{M}_{y}(\mathcal{U})$, and assume $\mu$ is Borel-definable over $A \subset \mathcal{U}$. Then the \textbf{Morley product}, denoted $\mu \otimes \nu$, is the unique Keisler measure in $\mathfrak{M}_{xy}(\mathcal{U})$ such that for any $\cL_{\cU}$-formula $\varphi(x,y,b)$, 
\[
    (\mu \otimes \nu)(\varphi(x,y,b)) = \int_{S_{y}(Ab)} F_{\mu,Ab}^{\varphi}\, d(\nu|_{Ab})
\]
where $\nu|_{Ab}$ is the restriction of $\nu$ to $S_y(Ab)$. 

Given a sequence of variables $x_{<\omega}$, each of sort $x$, we  define $\mu^{(1)}(x_0)= \mu(x_0)$, $\mu^{(n+1)}(x_0,\ldots,x_n) = \mu(x_0) \otimes \mu^{(n)}(x_1,\ldots,x_n)$, and $\mu^{(\omega)}(x_{<\omega}) = \bigcup_{n < \omega} \mu^{(n)}(x_{<n})$. 
\end{definition}

\begin{definition}[\fim]\label{def:fimDL}
We say that $\mu\in\kM_x(\cU)$ is a \textbf{frequency interpretation measure} (or \textbf{\textit{fim}}) if there is $M\prec\cU$ such that for any $\mathcal{L}$-formula $\varphi(x,y)$, there are consistent $\cL_M$-formulas $(\theta(x_1,\ldots,x_n))_{n\geq 1}$ satisfying the following properties:
\begin{enumerate}[$(i)$]
\item For any $\epsilon>0$, there is a $k(\epsilon)\geq 1$ such that if $n\geq k(\epsilon)$ and $\cU\models\theta_n(\abar)$ then $\mu(\varphi(x,b))\approx_{\epsilon} \Av(\abar)(\varphi(x,b))$ for any $b\in\cU^y$. 
\item $\lim_{n\to\infty}\mu^{(n)}(\theta_n(x_1,\ldots,x_n))=1$.
\end{enumerate}
In this case, we also say that $\mu$ is \textbf{\fim\ over $M$}.
\end{definition}

\begin{fact}
    \fim\ $\Rightarrow$ \famm\ $\Rightarrow$ \dfs.
\end{fact}

\begin{fact}[\cite{CoGa}]\label{fact:CG}
    An invariant type $p\in S_x(\cU)$ is generically stable over $M\prec\cU$ if and only if it is \fim\ over $M$.
\end{fact}


\section{Self-averaging measures}\label{sec:self-averaging}

Recall from the introduction our motivating question of whether \fim\ measures can be characterized by some property that ``clearly" resembles the various characterizations of generic stability for types given in Fact \ref{fact:genstab}.  Toward formulating this more rigorously, the first challenge we face is that there is no existing notion of a Morley sequence in a \emph{measure}. This motivates the following definition, which provides one possibility for such a notion.

\begin{definition}[Keisler-Morley measure]\label{defn:fim-measure}
Suppose  $\mu\in\kM_x(\cU)$ is Borel-definable over $M\prec\cU$, and let $x_{<\omega}$ be a sequence of variables with each $x_i$ of sort $x$.  Then $\lambda \in \mathfrak{M}_{x_{<\omega}}(\mathcal{U})$ is a \textbf{Keisler-Morley measure in $\mu$ over $M$} if $\lambda|_{M} = \mu^{(\omega)}|_{M}$. We let $\KM(\mu/M)$ denote the set of Keisler-Morley measures in $\mu$ over $M$. 
\end{definition}

Using Keisler-Morley measures as replacements of Morley sequences, 
we now formulate a Keisler measure analogue  of the characterization of generically stable types given in Fact \ref{fact:genstab}$(i)$. 

\begin{definition}[self-averaging]\label{def:GSM} Suppose $\mu \in \mathfrak{M}_{x}(\mathcal{U})$ is Borel-definable over $M\prec\cU$. Then $\mu$ is \textbf{self-averaging over $M$} if for any $\cL_{\cU}$-formula $\varphi(x)$ and any  $\lambda\in \KM(\mu/M)$, $\lim_{i\to\infty}\lambda(\varphi(x_i))=\mu(\varphi(x))$.
\end{definition}

Our next goal is to study properties of self-averaging measures, and to see how this notion fits with other notions of ``stable-like" behavior. First, we will prove that any self-averaging measure is automatically ``uniformly self-averaging", which establishes the analogue of Fact \ref{fact:genstab}[$(i)\Leftrightarrow (ii)$]. The proof will rely on the following topological lemma.  

\begin{lem}\label{lem:uniform-pruning}
Let $f\colon X\to K$ be an arbitrary function from a compact Hausdorff space $X$ to a compact interval $K\seq \R$. Suppose there is a closed subset $C\seq K^\omega\times X$ satisfying the following properties.
  \begin{enumerate}[\hspace{5pt}$\ast$]
  \item The projection of $C$ onto $X$ is all of $X$.
  \item If $(\alpha,x) \in C$ and $g\colon \omega \to \omega$ is strictly increasing, then $(\alpha \circ g,x) \in C$.
  \item For any $(\alpha,x) \in C$, $\lim_{i \to \infty}\alpha(i) = f(x)$.
  \end{enumerate}
  Then $f$ is continuous and, for any $\e > 0$, there is an $n_\e < \omega$ such that for any $(\alpha,x)\in C$,  $\{i<\omega : \alpha(i) \not\approx_\e f(x)\}| \leq n_\e$.
\end{lem}
\begin{proof}
We first prove $f$ is continuous. Suppose not.  Fix $x \in X$, a family $(y_j)_{j \in J}$ in $X$, and an ultrafilter $\Dc$ on $J$ such that $\lim_{j \to \Dc}y_j = x$ and $\lim_{j \to \Dc}f(y_j) \neq f(x)$. Let $r = \lim_{j \to \Dc} f(y_j)$. Fix $\e > 0$ with $|r-f(x)| > \e$.

  By throwing away part of $J$ if necessary, we may assume that $r \approx_{\frac{1}{3}\e} f(y_j)$ for all $j \in J$. For each $j \in J$, fix $(\alpha_j,y_j) \in C$. By assumption, we must have that $\lim_{i \to \infty}\alpha_j(i) = f(y_j)$. This means that on some final segment of $\omega$, $\alpha_j(i) \approx_{\frac{1}{3}\e} f(y_j)$. By composing with appropriately chosen $g\colon \omega \to \omega$, we may therefore assume that $\alpha_j(i) \approx_{\frac{1}{3}\e} f(y_j)$ (and therefore $\alpha_j(i) \approx_{\frac{2}{3}\e} r$) for each $i < \omega$ and $j \in J$.

  Let $\alpha = \lim_{j \to \Dc}\alpha_j$. Since $C$ is closed, we have that $(\alpha,x) \in C$, but we also have by construction that $|\alpha(i) - f(x)| > \frac{1}{3}\e$ for every $i<\omega$, which is a contradiction. Therefore we must have that $f$ is continuous.

Now suppose the latter statement fails. Then there is a $\delta > 0$ such that for any $n<\omega$, there is a $(\beta_n,z_n) \in C$ such that
  \[
|\{i<\omega : \beta_n(i) \not \approx_\delta f(z_n)\}| > n.
\]
By composing each $\beta_n$ with an appropriately chosen $g\colon \omega \to \omega$, we may assume  $ \beta_n(i) \not \approx_\delta f(z_n)$ for all $i<n$. Fix a non-principal ultrafilter $\Dc$ on $\omega$ and let $(\beta,z) = \lim_{n \to \Dc} (\beta_n,z_n)$. Since $C$ is closed, we have  $(\beta,z) \in C$. By construction, we have for each $i<\omega$ that $\beta(i) \not \approx_{\frac{1}{2}\delta}\lim_{n \to \Dc} f(z_n) = f(z)$ (the equality holds since $f$ is continuous). This contradicts our initial assumptions, and thus no such $\delta$ exists.
\end{proof}

\begin{corollary}\label{cor:sa-def}
Suppose $\mu\in\kM_x(\cU)$ is self-averaging over $M\prec\cU$. Then $\mu$ is definable over $M$.
\end{corollary}
\begin{proof}
Fix a formula $\varphi(x,y)$. Let $C$ be the subset of $[0,1]^\omega \times S_y(M)$ consisting of those points $(\alpha,q)$ such that for some $\lambda \in \mathrm{KM}(\mu/M)$ and some $b\models q$, $\alpha(i) = \lambda(\varphi(x_i,b))$ for all $i<\omega$. Note that  the projection of $C$ onto $S_y(M)$ is all of $S_y(M)$. It is also clearly the case that if $(\alpha,q) \in C$ and $g\colon \omega \to \omega$ is a strictly increasing function, then $(\alpha\circ g,q) \in C$ as well. Furthermore, since $\mu(x)$ is self-averaging, we have that for every $(\alpha,q) \in C$, $\lim_{i \to \infty}\alpha(i) = F_\mu^\varphi(q)$.\medskip

  \noindent\emph{Claim.} $C$ is a closed subset of $[0,1]^\omega \times S_y(M)$.

 \noindent \emph{Proof.} Suppose that $(\alpha,q)$ is an element of the closure of $C$. Fix $b \in \Uc$ with $\tp(b/M) = q(y)$. Now suppose that there is no $\lambda \in \mathrm{KM}(\mu/M)$ such that $\alpha(i) = \lambda(\varphi(x_i,b))$ for all $i<\omega$. Since $\mathrm{KM}(\mu/M)$ is a closed subset of $\frk{M}_{\xbar}(\Uc)$ and is the set of global extensions of a measure over $M$, we must have by compactness that there is an $\e > 0$, an $n<\omega$, and a $\psi(y) \in \tp(b/M)$ such that if $\lambda \in \mathrm{KM}(\mu/M)$ and $c \in \psi(\Uc)$, then $\alpha(i) \not \approx_\e \lambda(\varphi(x_i,c))$ for some $i<n$. But $(\alpha,q)$ is in the closure of $C$, so there must be some $\lambda \in \mathrm{KM}(\mu/M)$ and $c \in \psi(\Uc)$ such that $\alpha(i) \approx_\e \lambda(\varphi(x_i,c))$ for all $i<n$, which is a contradiction. \hfill $\square_{\text{claim}}$
 \medskip

  Now we can apply Lemma \ref{lem:uniform-pruning} to $C$ and the function $F_\mu^\varphi\colon S_y(M) \to [0,1]$, and conclude  that $F_\mu^\varphi$ is continuous. Since we can do this for any formula $\varphi(x,y)$, we have that $\mu$ is definable over $M$.
\end{proof}

We now prove the analogue of Fact \ref{fact:genstab}[$(i)\Leftrightarrow(ii)$] for self-averaging measures.

\begin{proposition}\label{prop:sa-unif}
Suppose $\mu\in\kM_x(\cU)$ is Borel-definable over $M\prec\cU$. Then $\mu$ is self-averaging over $M$ if and only if for any $\cL$-formula $\varphi(x,y)$ and any $\epsilon > 0$, there is some  $n_{\epsilon,\varphi}\geq 1$ such that for any $b \in \mathcal{U}^{y}$ and $\lambda \in \KM(\mu/M)$,
\[
    |\{i \in \omega: \lambda(\varphi(x_i,b)) \not\approx_\epsilon \mu(\varphi(x,b)) \}| < n_{\epsilon,\varphi}.
\]
\end{proposition}
\begin{proof}
The right-to-left direction is clear. For left-to-right, assume $\mu$ is self-averaging over $M$. Fix an  $\cL$-formula $\varphi(x,y)$. Recall that in the proof of Corollary \ref{cor:sa-def}, we showed that the assumptions of Lemma \ref{lem:uniform-pruning}  hold for the function $F^\varphi_\mu\colon S_y(M)\to [0,1]$ and the set $F$ of points $(\alpha,q)\in [0,1]^\omega\times S_y(M)$ such that for some $\lambda\in \KM(\mu/M)$ and $b\models q$, we have $\alpha(i)=\lambda(\varphi(x_i,b))$ for all $i<\omega$. So by Lemma \ref{lem:uniform-pruning}, we have that for each $\e > 0$, there is an $n_{\varphi,\e} < \omega$ such that for any $\lambda \in \mathrm{KM}(\mu/M)$ and any $b \in \Uc$, $|\{i < \omega : \lambda(\varphi(x_i,b)) \not \approx_\e F_\mu^\varphi(\tp(b/M))\}| \leq n_{\varphi,\e}$. Since we can do this for any $\varphi(x,y)$ and $\e >0$, we have that $\mu$ is self-averaging over $M$.
\end{proof}

At this point, the reader may wonder about condition $(iii)$ of Fact \ref{fact:genstab}. Recall that this condition (for a global type $p$) essentially says that there is no witness to the order property using a Morley sequence in $p$ (over a small model). We can obtain a reasonable analogue using Keisler-Morley measures in place of Morley sequences, together with a continuous version of the order property. However, in this case, we have only been able to prove an implication in one direction. 

\begin{proposition}\label{prop:GRASstable}
Suppose  $\mu \in \mathfrak{M}_{x}(\mathcal{U})$  is self-averaging over $M\prec\cU$. Then there does not exist an $\Lc$-formula $\varphi(x,y)$, a Keisler-Morley measure $\lambda\in\KM(\mu/M)$,  a sequence $(b_i)_{i<\omega}$ from $\Uc^y$, and real numbers $r<s$ in $[0,1]$,  such that if $i\leq j$ then $\lambda(\varphi(x_i,b_j))\leq r$ and if $i>j$ then $\lambda(\varphi(x_i,b_j))\geq s$.
\end{proposition}
\begin{proof}
Suppose there are such $\varphi(x,y)$, $(b_i)_{i<\omega}$, $r<s$, and $\lambda$. Fix some $\epsilon<\frac{s-r}{2}$. Apply Proposition \ref{prop:sa-unif} to $\varphi(x,y)$ and $\epsilon$ to get some integer $n\geq 1$. Let $b=b_{n+1}$. Then we see that there is no $t\in[0,1]$ such that $|\{i<\omega:|\lambda(\varphi(x_i,b))-t|>\epsilon\}|\leq n$ (which contradicts the choice of $n$). Indeed if $t\leq \frac{r+s}{2}$ then $|\lambda(\varphi(x_i,b))-t|>\epsilon$ for all $i>n+1$; and if $t>\frac{r+s}{2}$ then $|\lambda(\varphi(x_i,b))-t|>\epsilon$ for all $i\leq n+1$. 
\end{proof}

Once again, the converse of the previous proposition remains an open question. Indeed, it is not  clear whether the conclusion of the  proposition implies \emph{any} of the common notions of good behavior for measures (e.g., \dfs).

Next we turn to the main result of this section, which we state now (and then prove over the course of a few steps). 

\begin{theorem}\label{thm:fim-main}
Suppose $\mu\in\kM_x(\cU)$ is Borel definable over $M\prec\cU$.
\begin{enumerate}[$(a)$]
    \item If $\mu$ is \fim\ over $M$ then it is self-averaging over $M$.
    \item If $\mu$ is self-averaging over $M$ then it is \famm\ over $M$. 
\end{enumerate}
\end{theorem}

The proof of part $(a)$ will require the following fact about probability measures.

\begin{lem}\label{lem:limit}
  Let $\nu$ be a Borel probability measure on $[0,1]^\omega$, and fix $r\in [0,1]$. Suppose that for every $\e > 0$, there is some $n\geq 1$ such that for any $i_1<i_2<\dots <i_n<\omega$, 
  \[
    \int \left| r - \frac{1}{n} \sum_{k=1}^n x_{i_k} \right| d\nu < \e.
  \]
  Then $\lim_{i \to \infty} \int x_i\, d\nu = r$.
\end{lem}
\begin{proof}
Suppose not. Then for some $\e>0$, $\left|r -  \int x_i\, d\nu \right| \geq \e$ for infinitely many $i<\omega$. It must either be the case that for infinitely many $i < \omega$, $\int x_i\, d\nu \geq r + \e$ or that for infinitely many $i < \omega$, $\int x_i\, d\nu \leq r - \e$. We may assume without loss of generality that the first case occurs. Let $I \subseteq \omega$ be the infinite set of such indices. 

  Let $n$ be as in the statement of the lemma (for our choice of $\e$). Pick $i_1< i_2<\ldots < i_n$ from $I$. Then
  \begin{multline*} 
  \e \leq  \frac{1}{n} \sum_{k=1}^n \int (x_{i_k}-r)\, d\nu   
         =\left| \frac{1}{n} \sum_{k=1}^n \int (x_{i_k}-r)\, d\nu \right| 
         =  \left| \int \frac{1}{n}\sum_{k=1}^n (x_{i_k} - r)\, d\nu \right| \\
         \leq \int \left| \frac{1}{n}\sum_{k=1}^n(x_{i_k}-r) \right|d\nu  
    =  \int \left| r - \frac{1}{n} \sum_{k=1}^n x_{i_k} \right| d\nu < \e,
  \end{multline*}
  which is absurd.
\end{proof}

\begin{proof}[\textnormal{\textbf{Proof of Theorem \ref{thm:fim-main}}$(a)$}]
Suppose $\mu\in\kM_x(\cU)$ is fim over $M\prec\cU$. We want to show that $\mu$ is self-averaging over $M$. Fix some $\lambda\in \KM(\mu/M)$ and an $\cL_{\cU}$-formula $\varphi(x,b)$. Let  $r=\mu(\varphi(x,b))$. We need to show $\lim_{i\to\infty}\lambda(\varphi(x_i,b))=r$. Throughout the proof, we let $\xbar=(x_i)_{i<\omega}$, where each $x_i$ is of sort $x$.

\medskip

\noindent\textit{Claim.} For any $\epsilon>0$, there is  some $N=N_\epsilon \in \mathbb{N}$ such that for any $k > N$ and any $i_1<\ldots<i_k<\omega$, 
\[ 
\int|r - \Av(x_{i_1},\ldots, x_{i_k})(\varphi(x,b)) |\, d\lambda < \epsilon. \tag{$\ast$}
\]

\noindent\textit{Proof.} Fix $\epsilon>0$. Since $\mu$ is \fim\ over $M$, there is some $N$ such that for any $k > N$, we have  an $\cL_M$-formula $\theta_k(x_1,\ldots,x_k)$ satisfying the following properties:
\begin{enumerate}
\item[(1)] $\mu^{(k)}(\theta_k(x_1,\ldots,x_k)) > 1 - \frac{\epsilon}{3}$. 
\item[(2)] If $\cU\models \theta_k(\abar)$, then $ \sup_{c \in \mathcal{U}^{y}} |\Av(\abar)(\varphi(x,c)) - \mu(\varphi(x,c))| < \frac{\epsilon}{3}$. 
\end{enumerate} 
Moreover, since $\mu$ is definable over $M$, there is an $\cL_M$-formula $\psi(y)$ such that:
\begin{enumerate}
\item[(3)] $\mathcal{U} \models \psi(b)$.
\item[(4)] For any $c \in \mathcal{U}^{y}$, if $\mathcal{U} \models \psi(c)$ then  $|\mu(\varphi(x,c)) - r| < \frac{\epsilon}{3}$. 
\end{enumerate} 
Fix $k > N$ and $i_1<\ldots<i_k$. We need to prove $(\ast$). Define $f\colon \mathcal{U}^{\xbar} \to [0,1]$ via 
\[ 
f(\abar) = \sup_{c \in \psi(\mathcal{U}^{y})}|r - \Av(a_{i_1},\ldots,a_{i_k})(\varphi(x,c))|. 
\] 
  Note that if $\abar,\abar'\in\cU^{\xbar}$ have the same type over $M$, then $f(\abar)=f(\abar')$. So we have a well-defined map $g\colon S_{\xbar}(M) \to [0,1]$ via $g(p) = f(\abar)$ where $\abar \models p$. 
We claim $g$ is also continuous (and hence integrable). To see this, note that by construction, $g$ has finite image, say $\{r_1,\ldots,r_n\}$. For $1\leq i\leq n$, define the $\cL_M$-formula  
\[
\gamma_i(x_{i_1},\ldots,x_{i_k}) : = \exists y( \psi(y) \wedge |r - \Av(x_{i_1},\ldots,x_{i_k})(\varphi(x,y))| \geq r_i).
\]
Then for any $p\in S_{\xbar}(M)$ and $1\leq i\leq n$, $g(p) = r_i$ if and only if 
\[ 
\bigwedge_{j \leq i} \gamma_j(x_{i_1},\ldots,x_{i_k}) \wedge \bigwedge_{i < j} \neg \gamma_{j}(x_{i_1},\ldots,x_{i_k}) \in p. 
\]
It follows that $g$ is continuous. 

Now, note that if $\abar c \in \mathcal{U}^{x_{i_1}\ldots x_{i_k}y}$ and  $\cU\models \theta_k(\abar) \wedge \psi(c)$, then 
\[
|r - \Av(\abar)(\varphi(x,c))|  \leq |r - \mu(\varphi(x,c))| + |\mu(\varphi(x,c)) - \Av(\abar)(\varphi(x,c))| <   \frac{2\epsilon}{3}.
\]
It follows that for any $p \in S_{\xbar}(M)$,  if $\theta_{k}(x_{i_1},\ldots,x_{i_k}) \in p$ then  $g(p) \leq \frac{2\epsilon}{3}$. Therefore 
\begin{multline*} 
\int g\, d(\mu^{(\omega)}|_M)\leq \frac{2\epsilon}{3}\mu^{(\omega)}(\theta_k(x_{i_1},\ldots,x_{i_k}))+\mu^{(\omega)}(\neg\theta(x_{i_1},\ldots,x_{i_k}))\\
\leq \frac{2\epsilon}{3}+\mu^{(k)}(\neg\theta_k(x_1,\ldots,x_k)) < \frac{2\epsilon}{3} + \frac{\epsilon}{3} = \epsilon. 
\end{multline*} 
Since $\lambda|_M = \mu^{(\omega)}|_{M}$, it follows that $\int g\, d(\lambda|_M) < \epsilon$.

Now define $h\colon S_{\xbar}(\mathcal{U}) \to [0,1]$ via $h(p) = |r - \Av(a_{i_1},\ldots,a_{i_k})(\varphi(x,b))|$  where $\abar \models p|_{Mb}$. Recall that $\cU\models\psi(b)$, and so  for any $p \in S_{\xbar}(\mathcal{U})$, we have  $h(p) \leq g(p|_M)$. Hence, if  $\rho\colon S_{\xbar}(\mathcal{U}) \to S_{\xbar}(M)$ denotes the restriction map, then
\[
\int h\, d\lambda \leq \int g \circ \rho\, d(\lambda) = \int g\, d(\lambda|_M) < \epsilon. 
\]
This yields $(\ast)$ by definition of $h$. \hfill $\square_{\text{claim}}$\medskip

Now define $\Phi\colon S_{\xbar}(\mathcal{U}) \to [0,1]^{\omega}$ via $\Phi(p) = (\mathbf{1}_{\varphi(x_i,b)}(p))_{i \in \omega}$. Note that this map is continuous. Let $\nu$ be the Borel probability measure on $[0,1]^\omega$ obtained from the pushforward of $\lambda$ along $\Phi$, i.e., for any Borel set $B \subseteq [0,1]^{\omega}$, $\nu(B) = \lambda(\Phi^{-1}(B))$. For any $\epsilon>0$ and $k>N_\epsilon$, if $i_1<\ldots<i_k<\omega$ then 

\begin{multline*} 
\int {\textstyle \left| r - \frac{1}{k} \sum_{j=1}^k x_{i_j}\right|}\, d\nu = \int {\textstyle \left| r - \frac{1}{k}  \sum_{j=1}^k x_{i_j}\right| \circ \Phi}\, d\lambda\\
=  \int |r - \Av(x_{i_1},\ldots,x_{i_k})(\varphi(x,b))|\, d\lambda < \epsilon, 
\end{multline*} 
where the final inequality follows from the Claim.
Therefore $\lim_{i \to \infty} \int x_id\nu = r$ by Lemma \ref{lem:limit}. Hence we conclude that 
\[ 
 \lim_{i \to \infty} \lambda(\varphi(x_{i},b)) =\lim_{i\to \infty} \int x_i \circ \Phi\, d\lambda = \lim_{i \to \infty} \int x_i\, d\nu = r.\qedhere
\]
\end{proof}

Now we start toward  part $(b)$ of Theorem \ref{thm:fim-main}. To ease notation, we will continue to use $\bar{x}$ for an infinite sequence $(x_i)_{i<\omega}$ of variables, each of the same sort $x$. 

\begin{lemma}\label{lem:weird-extension}
 Suppose $\mu\in\kM_x(\cU)$ is invariant over $M\prec\cU$, and   $\sigma \in \frk{M}_{\xbar y}(M)$ is such that $\sigma |_{\xbar} = \mu^{(\omega)}|_M$ and $\sigma |_y$ is a type. Then there is  some $\lambda \in \KM(\mu/M)$ and  $b \in \Uc$ such that for any $\cL_M$-formula $\psi(\xbar,y)$, $\lambda(\psi(\xbar,b)) = \sigma(\varphi(\xbar,y))$.
\end{lemma}
\begin{proof}
Fix a $b \in \Uc$ satisfying $\sigma |_y$. Let $\lambda_0 \in \frk{M}_{\xbar y}(Ma)$ be the measure defined by $\lambda_0(\psi(\xbar,b)) = \sigma(\psi(\xbar,y))$ for every $M$-formula $\psi(\xbar,y)$. Then any global extension $\lambda$ of $\lambda_0$ has the required properties. 
\end{proof}

Given $M\prec\cU$, let $\rho_M\colon\kM_{\xbar}(\cU)\to \kM_{\xbar}(M)$ denote the restriction map.

\begin{prop}\label{prop:sa-char}
An $M$-invariant measure $\mu\in \frk{M}_x(\Uc)$ is self-averaging over $M$ if and only if for any formula $\varphi(x,y)$ and $\e > 0$, there is an open neighborhood $U$ of $\mu^{(\omega)}|_M$ in $\frk{M}_{\bar{x}}(M)$ and an $n\geq 1$ such that for all $\nu \in \rho_M\inv(U)$ and $b \in \Uc^y$, 
\[
\left|\mu(\varphi(x,b)) - \frac{1}{n}\sum_{i < n}\nu(\varphi(x_i,b))\right| < \e.
\]
\end{prop}
\begin{proof}
$(\Leftarrow)$: This implication is immediate, as $\mathrm{KM}(\mu/M) \subseteq \rho_M\inv(U)$ for any neighborhood $U$ of $ \mu^{(\omega)}|_M$ in $\kM_{\xbar}(M)$.

$(\Rightarrow)$: Assume that $\mu$ is self-averaging. Fix some such $\varphi(x,y)$ and $\e > 0$, and let $n = n_{\varphi,\e}$ by as in Proposition \ref{prop:sa-unif}. Assume for the sake of contradiction that for every open neighborhood $U$ of $ \mu^{(\omega)}|_M$ in $\kM_{\xbar}(M)$, there are $\nu_U \in \rho_M\inv(U)$ and $b_U \in \Uc^y$ such that 
\[
\left|\mu(\varphi(x,b_U)) - \frac{1}{n}\sum_{i<n}\nu_U(\varphi(x_i,b_U))\right| \geq \e.
\]
Let $\sigma_U \in \frk{M}_{\xbar y}(M)$ be the measure defined by $\sigma_U(\psi(\xbar,y)) = \nu_U(\psi(\xbar,b_U))$ for all $\cL_M$-formulas $\psi(\xbar,y)$. Let $\mathcal{N}$ be the collection of neighborhoods of $\mu^{(\omega)}|_M$ in $\frk{M}_{\xbar}(M)$. Let $\mathcal{F}$ be an ultrafilter on $\mathcal{N}$ extending the filter generated by sets of the form $\{V \in \mathcal{N} : V \subseteq U\}$. Let $\sigma_{\mathcal{F}} = \lim_{U \to \mathcal{F}} \sigma_U\in\kM_{\xbar y}(M)$. By construction we have $\sigma_{\mathcal{F}}|_{\xbar} = \mu^{(\omega)}|_M$, and $\sigma_{\mathcal{F}}|_y$ is a type rather than just a measure.

Since $\mu$ is definable (by Corollary \ref{cor:sa-def}), the function $F^{\varphi}_\mu\colon S_y(M) \to [0,1]$ is continuous. This implies that
\[
F^{\varphi}_{\mu}(\sigma_{\mathcal{F}}|_y) = \lim_{U \to \mathcal{F}} F^{\varphi}_{\mu}(\tp(b_U/M)) = \lim_{U\to \mathcal{F}} \mu(\varphi(x,b_U)).
\]
Likewise, the function $\nu \mapsto \frac{1}{n}\sum_{i<n} \nu(\varphi(x_i,y))$ from $\frk{M}_{\xbar y}(M)$ to $[0,1]$ is clearly continuous, which implies that 
\[
\frac{1}{n}\sum_{i<n}\sigma_{\mathcal{F}}\left(\varphi(x_i,y)\right) = \lim_{U \to \mathcal{F}} \frac{1}{n}\sum_{i<n} \nu_U(\varphi(x_i,b_U)).
\]
Since
\[
\left|F^\varphi_\mu(\tp(b_U/M)) - \frac{1}{n}\sum_{i<n} \nu_U(\varphi(x_i,b_U))\right| \geq \e 
\]
for every $U \in \mathcal{N}$, we must have that 
\[
\left|F_\mu^\varphi(\sigma_{\mathcal{F}} |_y) - \frac{1}{n}\sum_{i<n} \sigma_{\mathcal{F}}(\varphi(x_i,y)) \right|\geq \e
\]
as well.
Apply \cref{lem:weird-extension} to $\sigma_{\mathcal{F}}$ to get $\lambda \in \KM(\mu/M)$ and $b \in \Uc^y$ such that $\lambda(\psi(\xbar,b)) = \sigma_{\mathcal{F}}(\psi(\xbar,y))$ for every $\cL_M$-formula $\psi(\xbar,y)$. This property implies that
\[
\left|\mu(\varphi(x,b)) - \frac{1}{n}\sum_{i<n} \lambda(\varphi(x_i,b))\right| \geq \e,
\]
which contradicts the fact that $\mu$ is self-averaging.
\end{proof}

\begin{proof}[\textnormal{\textbf{Proof of Theorem \ref{thm:fim-main}$(b)$}}]
Suppose $\mu\in\kM_x(\cU)$ is self-averaging over $M\prec\cU$. We want to show that $\mu$ is \famm\ over $M$. 
Fix a formula $\varphi(x,y)$ and an $\e > 0$. Recall that $\bar{x}=(x_i)_{i<\omega}$ where each $x_i$ is of sort $x$. By \cref{prop:sa-char}, there is an open neighborhood $U$ of $\mu^{(\omega)}|_M$ in $\frk{M}_{\xbar}(M)$  such that for any $\nu \in \rho_M\inv(U)$ and any $b \in \Uc^y$, 
\[
\mu(\varphi(x,b)) \approx_\epsilon \frac{1}{n}\sum_{i<n} \nu(\varphi(x_i,b))\tag{$\dagger$}
\]

The collection of measures of the form $\Av(\abar^0,\abar^1,\dots,\abar^{k-1})$, with each $\abar^i$ an $\omega$-tuple in $M^x$, is dense in $\frk{M}_{\xbar}(M)$.
Choose some such $\Av(\abar^0,\abar^1,\dots,\abar^{k-1})$ in $U$. Expanding out definitions, we get that for any $b \in \Uc^y$,
\[ 
\frac{1}{n}\sum_{i<n}\Av(\abar^0,\abar^1,\dots,\abar^{k-1})(\varphi(x_i,b))= \frac{1}{n}\sum_{i<n}\frac{1}{k}\sum_{j<k} \varphi(a_i^j,b). 
\]
Thus, using $\nu=\Av(\abar^0,\abar^1,\dots,\abar^{k-1})$ in $(\dagger)$, we see that
 the $nk$-tuple $(a_i^j)_{i<n,j<k}$ is a \famm\ approximation of $\mu$ for the formula $\varphi(x,y)$ with accuracy $\e$.
\end{proof}

\begin{remark}
    It follows from Theorem \ref{thm:fim-main}$(b)$ and \cite[Proposition 5.17]{CoGaHa} that self-averaging measures commute with definable measures. In particular, self-averaging implies self-commuting. 
\end{remark}

It is worth pointing out that self-averaging does \emph{not} coincide with \famm\ (e.g., by Corollary \ref{cor:fimtypes} below and the fact that there are \famm\ types that are not generically stable \cite{CoGa}). Thus the main open question at this point is whether self-averaging is the same as \fim, or is a new notion properly between \fim\ and \famm. We have so far been unable to answer this question. However, we can conclude that \fim\ and self-averaging coincide for \emph{types}. 

\begin{corollary}\label{cor:fimtypes}
Suppose $p\in S_x(\cU)$. Then $p$ is generically stable over $M\prec\cU$ if and only if it is self-averaging over $M$.
\end{corollary}
\begin{proof}
First, if $p$ is self-averaging over $M$ then condition $(i)$ of Fact \ref{fact:genstab} clearly holds. Conversely, if $p$ is generically stable over $M$, then it is \fim\ over $M$ by Fact \ref{fact:CG}, and thus self-averaging over $M$ by Theorem \ref{thm:fim-main}$(a)$. 
\end{proof}

We also point out that \fim\ and self-averaging do coincide for  Keisler measures in NIP theories. 

\begin{corollary}\label{cor:saNIP}
Assume $T$ is NIP, and suppose $\mu\in\kM_x(\cU)$. Then $\mu$ is \fim\ over $M\prec\cU$ 
if and only if it is self-averaging over $M$.
\end{corollary}
\begin{proof}
Combine Theorem \ref{thm:fim-main} with the fact that \fim\ coincides with \dfs\  in NIP theories (see \cite[Theorem 3.2]{HPS}).
\end{proof}

\begin{remark} 
Remark 7.28 in \cite{Sibook} suggests a shorter proof of Theorem \ref{thm:fim-main}$(a)$. However, as was observed by Artem Chernikov, it is unclear whether the strategy is sound. Let us sketch the proposed proof. Assume $\mu \in \mathfrak{M}_{x}(\mathcal{U})$ is \fim\ over $M$. Fix an $\cL_{\cU}$-formula $\varphi(x,b)$ and $\lambda \in \KM(\mu/M)$.  Suppose toward a contradiction that $\lim_{i \to \infty} \lambda(\varphi(x_i,b)) \neq r\coloneqq\mu(\varphi(x,b))$. Without loss of generality, one may assume that there exists some $\epsilon>0$ such that $\lambda(\varphi(x_i,b)) > r + \epsilon$ for each $i \in \mathbb{N}$. 

Let $\theta_n(x_1,\ldots,x_n)_{n\geq 1}$ be a sequence of $\cL_M$-formulas witnessing \fim\ for $\mu$ with respect to $\varphi(x,y)$. Consider also the family of $\cL_{b}$-formulas $(\zeta_{n}(x_1,\ldots,x_n))_{n  \geq 1}$ where  
\[
    \zeta_{n}(x_1,\ldots,x_n) := |\Av(x_1,\ldots,x_n)(\varphi(x,b)) - r| > 
    \frac{\epsilon}{2}. 
\] 
 For $n$ large enough, we have $\Av(\abar)(\varphi(x,b))\approx_{\epsilon/2}r$ for any $\abar\models\theta_n(\xbar)$, which implies that $\zeta_n(\xbar)\wedge\theta_n(\xbar)$ is inconsistent. Since $\lambda\in\KM(\mu/M)$, we also have $\lim_{n\to\infty}\lambda(\theta_n(\xbar))=\lim_{n\to\infty}\mu^{(n)}(\theta_n(\xbar))=1$. Thus one would obtain a contradiction by proving that $\lim_{n\to\infty}\lambda(\zeta_n(\xbar))>0$. In \cite[Remark 7.28]{Sibook}, this done via an appeal to the weak law of large numbers (see \cite[Proposition B.4]{Sibook}) to conclude that, in fact, $\lim_{n \to \infty} \lambda(\zeta_{n}(\xbar)) = 1$. However, as $\lambda$ is not a product (or some kind of amalgam) of measures, it is not clear that the weak law of large numbers applies in this context.
\end{remark}

\section{Randomizations of \fim\ measures}\label{sec:randomization}

In this section, we study the relationship between \fim\ measures in a discrete\footnote{The assumption that $T$ is discrete should not be necessary, but we will use it for the sake of exposition.} theory $T$ and generically stable types in the  randomization $T^R$ of $T$. We first lay out some preliminaries regarding generically stable types and \fim\ measures in continuous logic. We then discuss transferring global definable measures over a monster model of a discrete theory $T$ to global types over a monster model of the  randomization $T^{R}$. This process was first developed in \cite{BYT} and we take the opportunity to flesh out some of the details. We conclude by showing that \fim\ measures transfer to generically stable types, although the converse remains open. We also prove that \famm\ measures transfer to \famm\ types.

\subsection{Continuous logic preliminaries}\label{sec:GScont}
Let $T$ be a continuous theory in a language $\mathcal{L}$ and $\mathcal{U}$ a monster model of $T$. We first  recall the definition of generic stability for continuous logic (which was first written explicitly by Khanaki in  \cite[Def.~2.11]{Khan}). In order to obtain a notion amenable to quantifier elimination, we will relativize the definition to a specific formula.

\begin{defn}\label{def:genstab} Suppose $p \in S_{x}(\mathcal{U})$ is invariant over $M\prec\cU$. Given an $\cL$-formula $\varphi(x,y)$, we say that $p$ is \textbf{generically stable (over $M$) relative to $\varphi(x,y)$} if there does not exist a Morley sequence $(a_i)_{i<\omega}$ in $p$ over $M$, a sequence $(b_i)_{i<\omega}$ in $\cU^y$, and real numbers $r<s$ such that $\varphi(a_i,b_j) \leq r$ if $i\leq j$ and $\varphi(a_i,b_j) \geq s$ if $i > j$. 

We say that $p$ is \textbf{generically stable (over $M$)} if it is generically stable (over $M$) relative to every $\cL$-formula $\varphi(x,y)$. 
\end{defn}

It is worth emphasizing that the previous definition involves Morley sequences, and so we need to assume $p$ is fully invariant (rather than invariant only for instances of $\varphi(x,y)$).

Now we turn to measures. For clarity, we recall the definition of a Keisler measure in the continuous context. 

\begin{definition} Fix $A \subseteq \mathcal{U}$. A \textbf{Keisler measure (in variable(s) $x$) over $A$} is a regular Borel probability measure on $S_{x}(A)$. As usual, we denote the collection of Keisler measures over $A$ in variable(s) $x$ as $\mathfrak{M}_{x}(A)$. Given some $\mu\in\kM_x(A)$ and an $\cL_A$-formula $\varphi(x)$, we often write $\int \varphi(x)\, d\mu$ as $\mu(\varphi(x))$.   
\end{definition}

The notions of invariance and definability for types and Keisler measures in continuous theories can be formulated identically to the discrete case (see Section \ref{sec:gloss}). For example, if $\mu\in\kM_x(\cU)$ is invariant over $A\subset\cU$ then for any $\varphi(x,y)$, we have the function $F^\varphi_{\mu,A}\colon S_y(A)\to [0,1]$ sending $q\in S_y(A)$ to $\mu(\varphi(x,b))$ for $b\models q$.

The next definition generalizes the notion of \fim\ for discrete theories, except that we again relativize to a single formula.

\begin{defn}\label{def:fimCL} Suppose $\mu\in\kM_x(\cU)$ is invariant over $M\prec\cU$, and further assume that the iterated Morley products $\mu^{(n)}$ are well-defined (e.g., this holds if $\mu$ is definable). Given an $\cL$-formula $\varphi(x,y)$, we say  $\mu\in\kM_x(\cU)$ is \textbf{\fim\ (over $M$) relative to $\varphi(x,y)$} if  there is a sequence $(\theta_n(x_1,\dots,x_n))_{n=1}^\infty$ of $\Lc_M$-formulas satisfying the following properties. 
\begin{enumerate}
    \item For any $\e > 0$, there is a $k(\e) \geq 1$ such that if $n \geq k(\e)$ and $ \theta_n(\abar) < 1$,
    \[
        \sup_{b \in \mathcal{U}^{y}}|\Av(\overline{a})(\varphi(x,b)) - F_{\mu,M}^{\varphi}(b) | \leq \epsilon. 
   \]
    \item $\lim_{n\to \infty} \mu^{(n)}(\theta_n <1) = 1$.
\end{enumerate}

We say that $\mu$ is a \textbf{frequency interpretation measure} (or a \textbf{\fim\ measure}) if it is \fim\ relative to every $\cL$-formula $\varphi(x,y)$. 
\end{defn}

Note that in the previous definition, the assumption on Morley products of $\mu$ was only needed to make sense of $(2)$. So we point out that, as in the discrete case, if for \emph{every} $\cL$-formula $\varphi(x,y)$, there is a sequence $(\theta_n(\xbar))_{n=1}^\infty$ satisfying $(1)$  with $\inf_{\xbar}\theta_n(\xbar)<1$ for all $n$, then it follows that $\mu$ is definable (in fact, finitely approximated in $M$), and hence $\mu^{(n)}$ is well-defined for all $n$.

Many of the basic results on generically stable types and \fim\ measures in discrete theories described in Section \ref{sec:gstypes-def} transfer directly to continuous logic. We state here only the results that we will need. 

\begin{fact}\label{fact:fimgs-cont}
Let $\varphi(x,y)$ be an $\cL$-formula and fix $M\prec\cU$.
\begin{enumerate}[$(a)$]
\item If $\mu\in\kM_x(\cU)$ is \fim\ over $M$ relative to $\varphi(x,y)$, then $F^\varphi_{\mu,M}$ is continuous.
\item Given an $M$-invariant type $p \in S_x(\cU)$, the following are equivalent.
\begin{enumerate}[$(i)$]
\item $p$ is generically stable over $M$ relative to $\varphi(x,y)$.
\item For any Morley sequence $(a_i)_{i<\omega}$ in $p$ over $M$ and any $b\in\cU^y$, 
\[
\lim_{i\to\infty}\varphi(a_i,b)=F^\varphi_{p,M}(b).
\] 
\item For any $\epsilon > 0$ there is some $n_{\epsilon}$ such that for any Morley sequence $(a_i)_{i < \omega}$ in $p$ over $M$, and any $b \in \mathcal{U}^{y}$, 
    $|\{i < \omega: |\varphi(a_i,b) - F_{p,M}^{\varphi}(b)| \geq \epsilon \}| \leq n_{\epsilon}$. 
\item $p$ is \fim\ over $M$ relative to $\varphi(x,y)$. 
\end{enumerate}
\end{enumerate}
\end{fact}
\begin{proof}[Proof.] 
Part $(a)$ follows directly from the fact that uniform limits of continuous functions are continuous. We now sketch the proof of part $(b)$. 

$(i)\Rightarrow (ii)$. This is a standard exercise in compactness. 

$(ii)\Rightarrow (iii)$.  Assume $(ii)$. We show that the map $F_{p,M}^{\varphi}\colon S_{y}(M) \to [0,1]$ satisfies the conditions of Lemma \ref{lem:uniform-pruning}. Let $C \subseteq [0,1]^\omega \times S_{y}(M)$ where 
   \begin{equation*}
       C = \left\{\Big((\varphi(a_i,b))_{i < \omega}, \tp(b/M)\Big): (a_i)_{i < \omega} \models p^{(\omega)}|_{M},~ b \in \mathcal{U}^{y} \right\}. 
   \end{equation*}
  We will show  that $C$ is closed. The rest of the conditions in Lemma \ref{lem:uniform-pruning} are easily verified by basic properties of Morley sequences. Let $\xbar=(x_i)_{i<\omega}$ and set  
\begin{equation*}
    A \coloneqq \{q \in S_{\bar{x}y}(\mathcal{U}): q|_{M,\xbar} = p^{\omega}|_{M}\}.  
\end{equation*}
Define $h\colon S_{\bar{x}y}(\mathcal{U}) \to [0,1]^{\omega} \times S_{y}(M)$ by $h(q) = ((\varphi(x_i,y)^{q})_{i < \omega},q|_{M,\xbar})$. Note $h$ is the product of continuous maps and hence continuous. Since both the domain and codomain are compact Hausdorff spaces, we conclude that $h(A)$ is closed. Since $h(A) = C$, we are finished. 

$(iii)\Rightarrow (i)$. Suppose $(i)$ fails. Then we have a Morley sequence $(a_i)_{i<\omega}$ in $p$, a sequence $(b_i)_{i<\omega}$ in $\cU^y$, and real numbers $r<s$ such that $\varphi(a_i,b_j)\leq r$ if $i\leq j$ and $\varphi(a_i,b_j)\geq s$ if $i>j$. Let $n_\epsilon$ be as in $(iii)$, with $\epsilon=\frac{s-r}{2}$. Let $b=b_{n_\epsilon}$. Then $\varphi(a_i,b)\leq r$ for all $i\leq n_\epsilon$, hence $F^\varphi_{p,M}(b)<r+\epsilon$ by $(iii)$. But then for any $i>n_\epsilon$, we have $\varphi(a_i,b)\geq s$, hence $|\varphi(a_i,b)-F^\varphi_{p,M}(b)|\geq\epsilon$, contradicting $(iii)$.

$(iii) \Rightarrow (iv)$. By $(iii)$, we have that $F^\varphi_{p,M}$ is continuous. Fix a Morley sequence $(a_i)_{i < \omega}$ and for each $n$, let $k_n$ be the largest integer such that 
\begin{equation*}
    \sup_{b \in \mathcal{U}^{y}}|\Av(a_1,\ldots,a_n)(\varphi(x,b)) - F_{p,M}^{\varphi}(b)| < \frac{1}{k_n}. 
\end{equation*}
We claim that the sequence $(k_n)_{n <\omega}$ goes to infinity. Indeed, one can see this by using condition $(iii)$ and observing that for any $b \in \mathcal{U}^{y}$,
\begin{equation*}
    |\Av(a_1,\ldots,a_n)(\varphi(x,b)) - F_{p,M}^{\varphi}(b)| \leq \sum_{i=1}^{n} \frac{1}{n}|\varphi(a_i,b) - F_{p,M}^{\varphi}(b)|. 
\end{equation*}
Setting $\theta_{n}(\bar{x}) \coloneqq \min\{1, k_n \cdot |\Av(\bar{x})(\varphi(x,y)) - F_{p,M}^{\varphi}(y)| \}$ it is straightforward to check  that the sequence $(\theta_{n}(\bar{x}))_{n \geq 1}$ has the properties necessary to conclude $(iv)$. 

$(iv)\Rightarrow  (iii)$. This is a routine exercise. 
\end{proof}

Given an $A$-invariant measure $\mu\in\kM_x(\cU)$, if $F^\varphi_{\mu,A}$ is continuous (e.g., if $\mu$ is definable over $A$), then $F^\varphi_{\mu,A}$  can be regarded as a formula in the sense of continuous logic. Thus we will often suppress the parameter set $A$ from the notation when it is clear or irrelevant. So, for example, given an invariant measure $\mu\in \kM_x(\cU)$, we will say  ``$F^\varphi_\mu$ is continuous" to mean that $F^\varphi_{\mu,A}$ is continuous for some/any $A\subset\cU$ for which $\mu$ is $A$-invariant. Moreover, we will treat $F^\varphi_\mu$ as a continuous formula in this case.

We now prove a technical lemma about \fim\ measures which does not have an explicit discrete analogue in previous literature. It will be convenient to have the following notation.

\begin{defn} 
Let $\varphi(x,y)$ be an $\cL$-formula and suppose $\mu\in\kM_x(\cU)$ is an invariant measure such that $F^\varphi_{\mu}$ is continuous. Given $n\geq 1$,  let 
  \[
    \chi^{\varphi}_{\mu,n}(x_1,\dots,x_n) = {\textstyle \sup_y }| \Av(x_1,\dots,x_n)(\varphi(x,y)) - F^{\varphi}_\mu(y)|.
  \]
  where each $x_i$ is of sort $x$.
\end{defn}

The formula $\chi_{\mu,n}^\varphi(\bar{x})$ measures how well the $n$-tuple $\bar{x}$ approximates the behavior of $\mu(\varphi(x,y))$ for arbitrary parameters $y$.

\begin{lem}\label{lem:fim-char-cont}
Let $\varphi(x,y)$ be an $\cL$-formula. Suppose $\mu\in\kM_x(\cU)$ is an invariant measure such that $\mu^{(n)}$ is well-defined for all $n\geq 1$ and $F^\varphi_{\mu}$ is continuous. Then $\mu$ is \fim\ relative to $\varphi(x,y)$ if and only if
  \[ 
  \lim_{n \to \infty} \int\chi^\varphi_{\mu,n}(\bar{x})\, d\mu^{(n)} = 0. 
  \]

\end{lem}
\begin{proof}
Let $(\theta_n(\xbar))_{n=1}^\infty$ be a sequence of formulas witnessing that $\mu$ is \fim\ relative to $\varphi(x,y)$, as given in \cref{def:fimCL}. For any given $\e > 0$, we can find a $k$ such that for all $n \geq k$, the following properties hold:
  \begin{enumerate}[$(i)$]
  \item If $\theta_n(\abar)<1$, then $|\Av(\abar) (\varphi(x,b)) - F^\varphi_{\mu}(b)| < \frac{1}{2}\e$ for all $b \in \Uc^y$.
  \item $\mu^{(n)}(\theta_n(\xbar) < 1) > 1 - \frac{1}{2}\e$.
  \end{enumerate}
  For a given $\e > 0$, $(i)$  and $(ii)$ together imply that
  \begin{align*}
    \int\chi_{\mu,n}^\varphi(\xbar)\, d\mu^{(n)} &\leq \mu^{(n)}(\theta_n\geq 1) + \int_{[\theta_n < 1]}\chi_{\mu,n}^\varphi(\xbar)\, d\mu^{(n)} \\
     & < \frac{1}{2}\e + \frac{1}{2}\e \mu^{(n)}(\theta_n < 1) < \e. 
  \end{align*}
  Since $\e$ was arbitrary, we conclude that $\lim_{n\to \infty} \int\chi_{\mu,n}^\varphi(\xbar)\, d\mu^{(n)} = 0$.

  We now prove the converse.  Assume that
  \[
    \lim_{n\to \infty}\int \chi^n_\varphi(\xbar)\, d\mu^{(n)} = 0.
  \]
  For each $n$, let $r_n \coloneqq \int \chi^n_\varphi(\xbar) d\mu^{(n)}$. Find a sequence of numbers $(s_n)_{n=1}^\infty$ in $[0,1]$ such that for each $n$, if $r_n < 1$, then $\sqrt{r_n} < s_n$ and
   $\lim_{n\to \infty}s_n = 0$.
  For each $n$, if $r_n = 1$, let $\theta_n(\xbar)$ be $d(x_0,x_0)$. Otherwise if $r_n < 1$, find a $[0,1]$-valued formula $\theta_n(\xbar)$ satisfying that if $\chi^n_{\varphi}\leq \sqrt{r_n}$, then $\theta_n(\xbar) = 0$ and if $\chi^n_{\varphi(\xbar)}\geq s_n$, then $\theta_n(\xbar) = 1$.

  For each $n$ such that $r_n < 1$, the $\mu^{(n)}$-measure of the set of types satisfying $\chi^n_\varphi(\xbar) \leq \sqrt{r_n}$ must be at least $1-\sqrt{r_n}$ (otherwise $\int\chi^n_\varphi(\xbar)d\mu^{(n)}$ would be greater than $(\sqrt{r_n})^2 = r_n$). Therefore we must have that $\mu^{(n)}(\theta_n(\xbar)<1) \geq \mu^{(n)}(\theta_n(\xbar) = 0) \geq 1-\sqrt{r_n}$. On the other hand, for any $\abar$, if $\theta_n(\abar) < 1$ then $\chi^n_\varphi(\abar) < s_n$ and hence $|\Av(\abar)(\varphi(x,b)) - F^\varphi_\mu(b)| < s_n$. Since $\lim_{n\to\infty}s_n=0$,  we have the required behavior and $\mu$ is \fim.
 \end{proof}

We now shift gears to  randomizations.

\subsection{Transferring definable measures to the randomization}\label{sec:ran-def-meas}

Let $T$ be a complete first-order discrete theory in a language $\mathcal{L}$. We let $T^{R}$ denote the theory of the randomization. We assume basic familiarity with the construction of $T^R$; see \cite{BYKR} for full definitions and details.  Let  $(\Omega, \mathcal{B},\mathbb{P})$ denote a probability algebra with event space $\Omega$, $\sigma$-algebra $\mathcal{B}$, and probability measure $\mathbb{P}$. Then any model $M\models T$ gives rise to a model $M^{(\Omega,\mathcal{B},\mathbb{P})}$, which we denote $M^{\Omega}$, as follows.
First define 
\[ 
M_0^\Omega := \{f\colon \Omega \to M: f \text{ is } \mathcal{B}\text{-measurable and has finite image}\}.
\] 
Then $M^{\Omega}$ is constructed by first taking the metric completion of $M_0^\Omega$ and then identifying functions up to $\mathbb{P}$-measure $0$. By construction, $M_0^\Omega$ is a metrically dense (pre-)substructure of $M^{\Omega}$.

Now let $\cU$ be a monster model of $T$. The model $\Uc^{\Omega}$ is not saturated and so we will always think of $\Uc^{\Omega}$ as elementarily embedded in a monster model of $T^{R}$, i.e., $\Uc^{\Omega} \prec \mathcal{C}$. If $a \in \Uc$, we let $f_{a}$ denote the constant function with value $a$, i.e., $f_{a}$ is the equivalence class (up to measure $0$) of the map which sends $\Omega$ to the point $a$.  For any $A \subseteq \Uc$, we let $A_{c} = \{f_{a}: a \in A\}$. For any  $\mathcal{L}$-formula $\varphi(x_1,\ldots,x_{n})$, we let $\Eb[\varphi(x_1,\ldots,x_{n})]$ denote the corresponding formula in the randomization. This formula is evaluated on elements $\bar{h} = (h_1,\ldots,h_n)$ of $\mathcal{U}^{\Omega}_0$ via 
\[
    \Eb[\varphi(\bar{h})] = \Pb\left(\{t \in \Omega: \mathcal{U} \models \varphi(h_1(t),\ldots,h_n(t))\} \right), 
\]
and is extended to $\mathcal{U}^{\Omega}$ by uniform limits. Suppose that $\varphi(x,y) \in \mathcal{L}$, $A \subset \mathcal{C}$, $q \in S_{x}(A)$, and $\bar{g} \in A^{y}$. We then write $(\mathbb{E}[\varphi(x,\bar{g})])^{q}$ to denote the unique real number $r$ such that for any $h \models q$, $\mathbb{E}[\varphi(h,\bar{g})] = r$ (i.e., $(\Eb[\varphi(x,\gbar)])^q$ is the value of the formula $\Eb[\varphi(x,\gbar)]$ determined by the type $q$).

We recall the the following basic fact about randomizations.

\begin{fact}[\cite{BYKR}]\label{fact:randomizationsQE}
$T^R$ has quantifier elimination, i.e., for any $\bar{h}$ in $\mathcal{C}$, $\tp(\bar{h})$ is uniquely determined by the values of $\Eb[\varphi(\bar{h})]$ for formulas $\varphi(\bar{x})$ from the language of $T$.  
\end{fact}

In the following analysis, it will be useful to treat the discrete theory $T$ as a continuous theory. In particular, this means that we will be considering arbitrary continuous functions on type space as formulas in the sense of continuous logic. If $F\colon S_x(\Uc) \to \Rb$ is a continuous function, then for $\abar \in \Uc$, we will may write $F(\abar)$ to mean $F(\tp(\abar/\Uc))$. An important fact that is implicit in this is that continuous functions are closed under supremal and infimal quantification. In other words, if $F\colon S_{xy}(\Uc) \to \Rb$ is a continuous function, then there is a continuous function $G\colon S_{x}(\Uc) \to \Rb$ satisfying that $G(a) = \sup_{b \in \Uc} F(a,b)$ for any $a \in \Uc$ (and likewise for $\inf$). We will freely write such functions with expressions such as $\sup_y F(x,y)$ or $\inf_y F(x,y)$.

Note that for any $\hbarr = (h_1,\ldots,h_n) \in (\Uc_0^\Omega)^{n}$, there is a finite partition $\Ac$ of $\Omega$ with the property that $\Ac \subset \mathcal{B}$ and every $h_i$ is constant on each element of $\Ac$. Given such an $\hbarr$, $\Ac$, and $A \in \Ac$, we will write $\hbarr|_A$ for the tuple of constant values of the functions in $\hbarr$ on the set $A$. Notice that for each $A \in \mathcal{A}$, $\hbarr|_{A}$ is an element of $\mathcal{U}^{n}$.

\begin{lemma}\label{lem:randomization-of-continuous-function} For any continuous function $F\colon S_{x}(\Uc) \to \Rb$, there is a unique continuous function $\Eb[F(-)]\colon S_{x}(\Uc^\Omega) \to \Rb$ satisfying that for any $\hbarr \in (\Uc_0^\Omega)^n$,
      \[
    \Eb[F(\tp(\hbarr/\Uc^\Omega))] = \sum_{A \in \Ac} \mathbb{P}(A)F(\tp(\hbarr|_A/\Uc)), \tag{$\ast$}
      \]
      where $\Ac$ is a finite measurable partition of $\Omega$ on which the elements of $\hbarr$ are constant.
\end{lemma}

\begin{proof}
For any $\e > 0$, we can find a finite family $(\varphi_i(x,\bbar_i))_{i < n}$ of $\mathcal{L}_{\mathcal{U}}$-formulas and a sequence $(r_i)_{i<n}$ of reals such that for any $p \in S_{x}(\Uc)$ 
\[
F(p) \approx_{\e} G_\e(p) \coloneqq \sum_{i<n} r_i \boldsymbol{1}_{\psi_i(x,\bbar_i)}(p).
\]
 We extend $G_\e$ to a function $\widehat{G}_\e$ on $S_{x}(\Uc^\Omega)$ by
\[
\widehat{G}_\e(p) \coloneqq \sum_{i<n}r_i (\Eb[\psi_i(x,\bar{f}_{b_i})])^p.
\]
By a standard computation, the limit $\lim_{\e \to 0}\widehat{G}_\e$ converges uniformly. We denote this limit by $\Eb[F(-)]\colon S_{x}(\Uc^\Omega) \to \Rb$. A direct computation shows that $(\ast)$ holds for this function. Uniqueness follows from the fact that $(\ast)$ specifies the value of the function on a dense set of points in $S_{x}(\Uc^\Omega)$.
\end{proof}

Note that if $\varphi(x,y)\colon S_{xy}(\mathcal{U}) \to [0,1]$ is a continuous function, then $\Eb[\varphi(x,y)]$ is a $\Uc^{\Omega}$-definable predicate in the randomization (in the sense of continuous logic) and can be safely regarded as a formula in the sense discussed above. The next observation will be used later.

\begin{fact}[\cite{BYKR}]\label{fact:randomizationsb}
 For any continuous map $\varphi(x,y):S_{xy}(\mathcal{U}) \to [0,1]$, we have that 
\begin{equation*}
    \mathcal{U}^{\Omega} \models \forall {x}\left(\textstyle \sup_{{y}} \Eb[\varphi({x},{y})] = \Eb\left[\sup_{{y}} \varphi({x},{y})\right] \right).
\end{equation*}
\end{fact}

\begin{proof}
    This is very similar to the proof of Proposition~2.5 in \cite{BYKR}. We provide  details. Furthermore, since the condition is equivalent to $\sup_x|(\sup_y\Eb[\varphi(x,y)]) - \Eb[\sup_y\varphi(x,y)]| = 0$, it is sufficient to verify this in the sub-pre-structure $\Uc^\Omega_0$.

For any $\gbar \in (\Uc^{\Omega}_0)^{x}$, there exists a finite measurable partition $\mathcal{A}$ of $\Omega$ such that each element of $\gbar$ is constant on each element of $\mathcal{A}$. For any $\epsilon > 0$ and $A \in \mathcal{A}$, there exists some $\bar{c}_{A} \in (\Uc)^{y}$ such that $\varphi(\bar{g}|_{A}, \cbar_{A}) > \sup_{y} \varphi(\bar{g}|_{A},y) - \epsilon$. Let $\cbar \in (\Uc^{\Omega}_0)^{y}$ such that $\bar{c}|_{A} = \cbar_{A}$ for each $A \in \mathcal{A}$. It is immediate that $\Eb[\varphi(\gbar,\cbar)] > \Eb[\sup_y\varphi(\gbar,y)]-\e$. Since we can do this for any $\e > 0$, we have that $\sup_{y} \Eb[\varphi(\gbar,{y})] \geq \Eb[\sup_{{y}} \varphi(\gbar,{y})]$ for any $\gbar \in (\Uc^\Omega_0)^{{x}}$.

For the other direction, we have that $ \sup_{xy} (\varphi({x},{y}) -\sup_{{z}}\varphi({x},{z})) \leq 0$, which implies that for any $\gbar\hbarr \in (\Uc_0^{\Omega})^{xy}$, $\Eb[\varphi(\gbar,\hbarr)] \leq \Eb[\sup_z \varphi(\gbar,z)]$ by the monotonicity of expected value. Therefore $\sup_y \Eb[\varphi(\gbar,y)] \leq \sup_y \Eb[\sup_z \varphi(\gbar,z)]$ by the monotonicity of supremum. Finally, $\sup_y \Eb[\sup_z \varphi(\gbar,z)] = \Eb[\sup_z \varphi(\gbar,z)]$, so $\sup_y \Eb[\varphi(\gbar,y)] \leq \Eb[\sup_z \varphi(\gbar,z)]$, as required.
\end{proof}

 Next we explicitly describe how to transfer a definable measure $\mu(x)$ over $\cU\models T$  to a definable type $r_{\mu}(x)$ over $\mathcal{C}$. This procedure was first established by Ben Yaacov in an unpublished research note \cite{BYT}.\footnote{Actually Ben Yaacov only assumes Borel-definability of $\mu$; see also Remark \ref{rem:BDbad}.} So we take the opportunity here to spell out the details.

\begin{fact}[{\cite[Prop.~1.1]{BYT}}]\label{fact:unique-definable-transfer}
   Let $\mu\in \mathfrak{M}_x(\Uc)$ be a definable measure. There is a unique $\Uc^\Omega$-definable type $r_\mu \in S_x(\Cc)$ defined as follows: for any $\varphi(x,\ybar) \in \Lc$, 
  \[
  F^{\Eb[\varphi(x,\ybar)]}_{r_\mu, \mathcal{U}^{\Omega}}= \Eb[F^{\varphi(x,\ybar)}_{\mu,\mathcal{U}}].
  \]   
    In particular, $(\Eb[\varphi(x,\bar{f_b})])^{r_\mu} = \mu(\varphi(x,\bbar))$ for every $\varphi(x,\bar{y}) \in \Lc$ and $\bar{b} \in \Uc^{\bar{y}}$.
\end{fact}

\begin{proof}
 In order to  construct $r_{\mu}$, we claim that it suffices to show that for any $\cL$-formulas $\varphi_0,\dots,\varphi_{n-1}$, $\bar{g}_0,\dots, \bar{g}_{n-1} \in (\Uc_0^\Omega)^{<\omega}$, and $\e>0$, there is some $q \in S_x(\Uc^\Omega)$ such that for each $i< n$,
\[
|(\Eb[\varphi_i(x,\bar{g}_i)])^q- \Eb[F_\mu^{\varphi_i(x,\bar{y})}(\bar{g}_i)]| < \e.
\]

Why?  
Satisfying the condition above yields a type $q_0 \in S_{x}(\mathcal{U}^{\Omega})$ such that for any $\mathcal{L}$-formula $\varphi(x,\ybar)$ and $\gbar \in (\mathcal{U}^{\Omega})^{y}$, $(\Eb[\varphi(x,\gbar)])^{q_0} = \Eb[F_\mu^{\varphi(x,\bar{y})}(\gbar)]$. By \cref{lem:randomization-of-continuous-function}, we have that for each $\varphi(x,\ybar)$, the function $\Eb[F_\mu^{\varphi(x,\bar{y})}]$ is well-defined and continuous. Therefore, by quantifier elimination, the  type $q_0$  extends uniquely to a complete $\cU^\Omega$-definable type over $\mathcal{C}$, which we denote $r_\mu$.

It remains to show the sufficient condition described above. Fix $\cL$-formulas $\varphi_0,\ldots,\varphi_{n-1}$, $\bar{g}_0,\dots, \bar{g}_{n-1} \in (\Uc_0^\Omega)^{<\omega}$, and $\e>0$. 
Let $G\subseteq \Uc$ be the (finite) union of the ranges of the elements which appear in the $\gbar_i$'s. By standard facts about measures, we can find a tuple $\abar = a_0,a_1,\dots, a_{m-1} \in \Uc^m$ with the property that for each $i<n$ and $b \in G$, $\mu(\varphi_i(x,b)) \approx_{\e} \Av(\abar)(\varphi_i(x,b))$.

Fix a finite measurable partition $\Ac$ of $\Omega$ with the property that every element of $G$ is constant on each element of $\Ac$. Then

\[
\Eb[F_\mu^{\varphi_i(x,\bar{y})}(\bar{g}_i)] =_{\hphantom{\e}} \sum_{A \in \Ac} \Pb(A)\mu(\varphi_i(x,\gbar_i|_A))  \approx_\e \sum_{A \in \Ac} \Pb(A)\Av(\abar)(\varphi_i(x,\gbar_i|_A)).
\]

We will now construct $q$ using $\Av(\abar)$. Let $I = [0,1)$ be the standard Lebesgue space. Let $\Omega \times I$ be the product measure space (with the product measure). We may regard $\Bc$ as a sub-$\sigma$-algebra of the $\sigma$-algebra of measurable subsets of $\Omega \times I$. Likewise, we may regard $\Uc^\Omega$ as a substructure of $\Uc^{\Omega \times I}$ by taking each $f(t) \in \Uc^\Omega$ and mapping it to $f'(t,s) \in \Uc^{\Omega \times I}$ where for each $(t,s) \in \Omega \times [0,1)$, $f'(t,s) = f(t)$. By quantifier elimination, this is an elementary embedding of $\Uc^\Omega$ into $\Uc^{\Omega \times I}$. Consider the function $h \in \Uc^{\Omega \times I}$ defined by $h(t,s) = a_j$ if $\frac{j}{m} \leq s < \frac{j+1}{m}$ for all $j<m$. Let $q = \tp(h/\Uc^\Omega)$.
  For each $i<n$, we have
 \begin{align*}
   (\Eb[\varphi_i(x,\gbar_i)])^q &_{\hphantom{\e}}=_{\hphantom{\e}} \sum_{j<m} \frac{1}{m}\Eb[\varphi_i(f_{a_j},\gbar_i)]\\ 
    &_{\hphantom{\e}}=_{\hphantom{\e}} \sum_{j<m} \frac{1}{m}\sum_{A \in \Ac}\Pb(A)\varphi_i(a_j,\gbar_i|_A) \\
    &_{\hphantom{\e}}=_{\hphantom{\e}} \sum_{A \in \Ac}\Pb(A)\sum_{j<m} \frac{1}{m}\varphi_i(a_j,\gbar_i|_A) \\
    &_{\hphantom{\e}}=_{\hphantom{\e}} \sum_{A \in \Ac}\Pb(A)\Av(\abar)(\varphi_i(x,\gbar_i|_A)) \\
    &_{\hphantom{\e}}\approx_\e \Eb[F_\mu^{\varphi_i(x,\bar{y})}(\bar{g}_i)].
 \end{align*}

 This finishes the construction of $r_\mu$. The final statement of the fact follows from the definition of $r_\mu$ and    \cref{lem:randomization-of-continuous-function}. 
\end{proof}

\begin{remark}\label{remark:invariance} Suppose $\mu \in \mathfrak{M}_{x}(\mathcal{U})$ and $\mu$ is definable over $M$. Then the type $r_{\mu}$ is actually definable over $M^{\Omega}$. This is more or less obvious, but we provide an explanation as an exercise in sanity. 

It suffices to show that $r_{\mu}$ is $M^{\Omega}$-invariant on formulas of the form $\mathbb{E}[\varphi(x,y)]$. Fix an $\cL$-formula $\varphi(x,y)$ and $g,h \in \mathcal{C}^{y}$ such that $\tp(g/M^{\Omega}) = \tp(h/M^{\Omega})$. In particular, we have $\mathbb{E}[\psi(h,f_{b})] = \mathbb{E}[\psi(g,f_{b})]$ for any $\mathcal{L}$-formula $\psi(x,y)$ and $b \in M^{y}$. It suffices to prove that $(\Eb[\varphi(x,g)])^{r_{\mu}} = (\Eb[\varphi(x,h)])^{r_{\mu}}$.

Since $\mu$ is definable over $M$, for every $\epsilon > 0$ there exists $\mathcal{L}_M$-formulas $\{\psi_i(y,b_i)\}_{i < n}$ and real numbers $r_1,\ldots,r_n$ such that 
\begin{equation*} 
\sup_{q \in S_{y}(\mathcal{U})}\left|F_{\mu}^{\varphi}(q) - \sum_{i \leq n} r_i \mathbf{1}_{\psi_i(y,b_i)}(q)\right| < \epsilon. 
\end{equation*} 
The $G_{\epsilon}$ in Lemma \ref{lem:randomization-of-continuous-function} can be constructed using the  $\{\psi_i(x,b_i)\}_{i < n}$ as defined above. Observe, 
\begin{multline*} 
\widehat{G}_{\epsilon}(\tp(h/\mathcal{U}^{\Omega})) = \sum_{i \leq n} r_i(\mathbb{E}[\psi_{i}(x,f_{b_i})])^{\tp(h/G)} = \sum_{i \leq n} r_i(\mathbb{E}[\psi_{i}(h,f_{b_i})]) \\ 
= \sum_{i \leq n} r_i(\mathbb{E}[\psi_{i}(g,f_{b_i})]) = \sum_{i \leq n} r_i(\mathbb{E}[\psi_{i}(x,f_{b_i})])^{\tp(g/G)} = \widehat{G}_{\epsilon}(\tp(g/\mathcal{U}^{\Omega})). 
\end{multline*} 
Now, we compute 
\begin{multline*} 
(\mathbb{E}[\varphi(x,h)])^{r_{\mu}} = F_{r_{\mu},\mathcal{U}^{\Omega}}^{\mathbb{E}[\varphi(x,y)]}(\tp(h/\mathcal{U}^{\Omega})) = \mathbb{E}[F_{\mu,\mathcal{U}}^{\varphi(x,y)}](\tp(h/\mathcal{U}^{\Omega}))\\
= \lim_{\epsilon \to 0} \widehat{G}_{\epsilon}(\tp(h/\mathcal{U}^{\Omega})) 
 =  \lim_{\epsilon \to 0} \widehat{G}_{\epsilon}(\tp(g/\mathcal{U}^{\Omega})) = \mathbb{E}[F_{\mu,\mathcal{U}}^{\varphi(x,y)}](\tp(g/\mathcal{U}^{\Omega}))\\ =  F_{r_{\mu},\mathcal{U}^{\Omega}}^{\mathbb{E}[\varphi(x,y)]}(\tp(g/\mathcal{U}^{\Omega})) = (\mathbb{E}[\varphi(x,g)])^{r_{\mu}} . 
\end{multline*} 
\end{remark} 

\begin{corollary}\label{cor:rmu-def} 
Let $\mu \in \mathfrak{M}_x(\mathcal{U})$ be a definable measure. Then for any continuous function $f\colon S_{x}(\mathcal{U}) \to \mathbb{R}$, we have  $\int f\, d\mu = (\Eb[f])^{r_{\mu}}$. 
\end{corollary}

\begin{proof} Fix $\epsilon>0$. Since $f$ is continuous, there are $\cL_{\cU}$-formulas $\psi_1(x),\ldots,\psi_n(x)$ and real numbers $s_1,\ldots,s_n$ such that $\lVert f - \sum_{i\leq n} s_i \psi_i \lVert_{\infty} < \epsilon$. Therefore, 
\begin{align*}
    \int f\, d\mu \approx_{\epsilon} \int \sum_{ i \leq n} s_i \mathbf{1}_{\psi_i}\, d\mu &= \sum_{i \leq n} s_i \mu(\psi_i(x)) = \sum_{i \leq n} s_i (\Eb[\psi_i(x)])^{r_\mu}\\
    &= \left(\Eb \left[\sum_{i \leq n} s_i (\mathbf{1}_{\psi_i}(x))\right]\right)^{r_\mu} \approx_{\epsilon} \left( \Eb[f] \right)^{r_{\mu}}. \qedhere
\end{align*}   
\end{proof}

The final goal of this section is to show that for definable measures, the transfer map $\mu\mapsto r_\mu$ commutes with Morley products. For this, we need two  lemmas.  

\begin{lemma}\label{random:1} Let $\mu \in \mathfrak{M}_{x}(\mathcal{U})$ be definable. Let $(h_i)_{i \in I}$ be a net in $(\cU_0^{\Omega})^{x}$ such that $\lim_{i \in I} \tp(h_i/\mathcal{U}^{\Omega}) = r_{\mu} |_{\Uc^\Omega}$. For each $i \in I$,  let $\mathcal{D}_i$ be a finite measurable partition of $\Omega$ such that $h_i$ is constant on each $D \in \mathcal{D}_i$. Set $\mu_i = \sum_{D \in \mathcal{D}_i} \mathbb{P}(D) \delta_{h_i|_{D}}$. Then $\lim_{i \in I} \mu_i = \mu$ (in the weak$^*$ topology). 
\end{lemma} 

\begin{proof} It suffices to show that for each $\cL_{\cU}$-formula $\varphi(x,\bbar)$, $\lim_{i \in I} \mu_i(\varphi(x,\bbar)) = \mu(\varphi(x,\bbar))$. Consider the following computation: 

\begin{align*}
    \lim_{i \in I} \mu_i(\varphi(x,\bbar)) &= \lim_{i \in I}
  \sum_{D \in \mathcal{D}_i} \mathbb{P}(D)\delta_{h_i|_{D}}(\varphi(x,\bbar)) \\
  &= \lim_{i \in I}
  \sum_{\substack{D \in \mathcal{D}_i~\textnormal{s.t.} \\ \Uc \models \varphi(h_i|_{D},\bbar)}} \mathbb{P}(D) \\
  &= \lim_{i \in I} \mathbb{P}(\{t \in \Omega: \Uc \models \varphi(h_i(t), \bbar) \}) \\
  &=  \lim_{i \in I} \mathbb{P}(\{t \in \Omega: \Uc \models \varphi(h_i(t),\bar{f_{b}}(t)) \})\\
  &=  \lim_{i \in I} \Eb[\varphi(h_i,\bar{f_{b}})] \\
  &= (\varphi(x,\bar{f_{b}}))^{r_\mu}\\
  & = \mu(\varphi(x,\bbar)). \qedhere
\end{align*}
\end{proof}

The next lemma is instrumental in finding good approximations of $r_{\mu}$ over $\mathcal{U}^{\Omega}$.

\begin{lem}\label{lem:independent-limit} Fix a definable measure $\mu \in \mathfrak{M}_{x}(\mathcal{U})$ and a finite, measurable partition $\mathcal{A}$ of $\Omega$. Then there is a net $(h_i)_{i \in I}$ of elements in $(\Uc^\Omega_0)^{x}$    and finite measurable partitions $(\Dc_i)_{i \in I}$ of $\Omega$ satisfying the following properties:
\begin{enumerate}[$(i)$]
    \item $\lim_{i \in I} \tp(h_i/\Uc^\Omega) = r_{\mu}|_{\mathcal{U}^{\Omega}}$. 
    \item $h_i$ is constant on each element of $\mathcal{D}_i$.
    \item  $\mathbb{P}(A \cap D) = \mathbb{P}(A) \mathbb{P}(D)$ for each $A \in \mathcal{A}$ and $D \in \mathcal{D}_i$. 
\end{enumerate}
\end{lem}
\begin{proof}
By quantifier elimination and uniform continuity, it suffices to show that for any $\e >0$, any finite sequence $(\varphi_i(x,\ybar))_{i<n}$ of $\Lc$-formulas, and any finite sequence $(\gbar_i)_{i<n}$ of tuples in $(\Uc^\Omega_0)^{\ybar}$, there is an $h \in \Uc^\Omega_0$ and a finite measurable partition $\Dc$ of $\Omega$ such that $\Eb[\varphi_i(h,\gbar_i)] \approx_\e (\Eb[\varphi_i(x,\gbar_i)])^{r_\mu}$, $h$ is constant on each element of $\Dc$, and $\Pb(A\cap D) = \Pb(A)\Pb(D)$ for each $A \in \Ac$ and $D \in \Dc$. 

So fix some such $\e$, $\varphi_i$'s, and $\gbar_i$'s. Let $G \subseteq \Uc$ be the collection of the ranges of the elements of the $\gbar_i$'s. Note that $G$ is a finite set. By basic facts about Keisler measures over models, we can find $a_0,\dots,a_{k-1}$ in $\Uc^{x}$ with the property that for each $i<k$ and each tuple $\cbar$ made from elements of  $G$, 
\[
\mu(\varphi_i(x,\cbar)) \approx_\e \Av(a_0,\dots,a_{k-1})(\varphi_i(x,\cbar)). 
\]
Let $\Ac^\ast$ be a finite, measurable partition refining $\Ac$ with the property that each element of $G$  is constant on each element of $\Ac^\ast$. Since $(\Omega,\Bc,\Pb)$ is an atomless probability algebra, we can find a finite partition $\Dc$ in $\Bc$ with $|\Dc| = k$, $\Pb(D) = \frac{1}{k}$ for each $D \in \Dc$, and $\Pb(A\cap D) = \Pb(A)\Pb(D)$ for every $A \in \Ac^\ast$ and $D \in \Dc$. Let $D_0,\dots,D_{k-1}$ be an enumeration of $\Dc$, and let $h$ be the element of $\Uc^\Omega_0$ with the property that $h(t) = a_j$ for any $t \in D_j$. We now have that for each $i<n$, 
\begin{align*}
   \Eb[\varphi_i(h,\gbar_i)] &=_{\hphantom{\e}} \sum_{A \in \Ac^\ast}\sum_{D \in \Dc} \Pb(A\cap D)  \mathbf{1}_{\varphi_i}(h|_{A\cap D},\gbar_i|_{A\cap D}) \\
   &=_{\hphantom{\e}} \sum_{A \in \Ac^\ast} \Pb(A)\sum_{D \in \Dc}\Pb(D)  \mathbf{1}_{\varphi_i}(h|_{D},\gbar_i|_{A}) \\
   &=_{\hphantom{\e}} \sum_{A \in \Ac^\ast} \Pb(A)\sum_{j<k}\frac{1}{k}  \mathbf{1}_{\varphi_i}(a_j,\gbar_i|_{A}) \\
   &=_{\hphantom{\e}} \sum_{A \in \Ac^\ast} \Pb(A) \Av(a_0,\dots,a_{k-1})(\varphi_i(x,\gbar_i|_{A})) \\
   &\approx_{\e} \sum_{A \in \Ac^\ast} \Pb(A) \nu(\varphi_i(x,\gbar_i|_{A})) \\
   &=_{\hphantom{\e}} (\Eb[\varphi_i(x,\gbar_i)])^{r_\mu}.\qedhere 
\end{align*}
\end{proof}

\begin{prop}\label{prop:prod-same-1}
    If $\mu\in\kM_x(\cU)$ and $\nu\in\kM_y(\cU)$ are definable, then 
    \[
    r_{\mu\otimes \nu}(x,y) = r_\mu(x)\otimes r_\nu(y).
    \]
\end{prop}
\begin{proof}
    Since the Morley product of definable types is definable, it is immediate that $r_\mu\otimes r_\nu$ is a definable type. Recall also that the Morley product of definable measures is definable, and so $r_{\mu\otimes\nu}$ is well-defined. By \cref{fact:unique-definable-transfer}, all we need to show is that for any $\varphi(x,y,\zbar)\in \Lc$, $F^{\Eb[\varphi(x,y,\zbar)]}_{r_\mu\otimes r_\nu} = F^{\Eb[\varphi(x,y,\zbar)]}_{r_{\mu \otimes \nu}}$ as functions on $S_{\zbar}(\Uc^\Omega)$. By continuity, it is sufficient to check this on realized types and moreover on elements of $\Uc^\Omega_0$. Fix a tuple $\gbar \in (\Uc^\Omega_0)^n$. Fix a finite, measurable partition $\Ac$  such that each element of $\gbar$ is constant on the elements of $\Ac$. Let $(h_i)_{i \in I}$ and $(\Dc_i)_{i \in I}$ be as in \cref{lem:independent-limit} with respect to $\nu$ and $\mathcal{A}$.  
    Consider the following computation:
    
\begin{align*}
      F^{\Eb[\varphi(x,y,\zbar)]}_{r_\mu \otimes r_\nu}(\tp(\gbar/\Uc^\Omega)) &=F^{\Eb[\varphi(x,y,\gbar)]}_{r_\mu}(r_{\nu}|_{\mathcal{U}^{\Omega}}) \\
      &=F^{\Eb[\varphi(x,y,\gbar)]}_{r_\mu}(\lim_{i \in I} \tp(h_i/\mathcal{U}^{\Omega})) \\
      &\overset{(a)}{=}\lim_{i \in I}F^{\Eb[\varphi(x,y,\gbar)]}_{r_\mu}( \tp(h_i/\mathcal{U}^{\Omega})) \\
      &= \lim_{i \in I} (\Eb[\varphi(x,h_i,\gbar)])^{r_\mu} \\
      &= \lim_{i \in I} \sum_{A \in \Ac}\sum_{D \in \Dc_i} \Pb(A\cap D) \mu(\varphi(x,h_i|_{A\cap D},\gbar|_{A\cap D})) \\
      &=  \sum_{A \in \Ac} \Pb(A) \lim_{i \in I}\sum_{D \in \Dc_i}\Pb(D) \mu(\varphi(x,h_i|_{ D},\gbar|_{A})) \\
      &=   \sum_{A \in \Ac} \Pb(A) \lim_{i \in I}\int F_\mu^{\varphi(x,y,\gbar|_A)}\, d\left(\sum_{D \in \Dc_i} \Pb(D)\delta_{h_i|_D}\right)  \\
      &\overset{(b)}{=}    \sum_{A \in \Ac} \Pb(A) \int F_\mu^{\varphi(x,y,\gbar|_A)}\, d\left(\lim_{i \in I}\sum_{D \in \Dc_i} \Pb(D)\delta_{h_i|_D}\right)  \\
      &\overset{(c)}{=}    \sum_{A \in \Ac} \Pb(A) \int F_\mu^{\varphi(x,y,\gbar|_A)}\, d\nu  \\
      &=    \sum_{A \in \Ac} \Pb(A) (\mu\otimes \nu)(\varphi(x,y,\gbar|_A))  \\
      &=    (\Eb[\varphi(x,y,\gbar)])^{r_{\mu\otimes \nu}}  \\
      &= F_{r_{\mu\otimes\nu}}^{\Eb[\varphi(x,y,\zbar)]}(\tp(\gbar/\Uc^\Omega)).
\end{align*} 

Equality $(a)$ follows from continuity and the fact that $r_\mu$ is definable. Equality $(b)$ follows from the fact that   $\int F^{\varphi(x,y,\gbar|_A)}_\mu d(-)$  is a continuous map from $\mathfrak{M}_{x}(\mathcal{U})$ to $[0,1]$. Equality $(c)$ follows from \cref{random:1}.
\end{proof}

\begin{corollary}\label{cor:prod-same}
  Suppose  $\mu \in \mathfrak{M}_{x}(\mathcal{U})$  is definable. Then for every $n\geq 1$, $r_{\mu^{(n)}}(\bar{x}) = (r_\mu)^{(n)}(\bar{x})$.
\end{corollary}
\begin{proof}
   This follows from \cref{prop:prod-same-1} by induction. 
\end{proof}

\begin{remark}\label{rem:BDbad}
As previously noted, Ben Yaacov \cite{BYT} defines the map $\mu\mapsto r_\mu$ only assuming Borel-definability of $\mu$. However, Proposition \ref{prop:prod-same-1} fails under a general Borel-definability assumption due to the fact that the Morley product need not preserve Borel-definability or satisfy associativity (even when all Morley products involved are Borel-definable). See \cite[Section 3]{CoGaHa} for details. On the other hand, if one assumes $T$ is countable and restricts to measures that are Borel-definable over a countable model, then the Morley product is associative and preserves Borel-definability. Thus one could potentially extend Proposition \ref{prop:prod-same-1} to Borel-definable measures under these additional countability assumptions.
\end{remark}

\subsection{Randomizations of \fim\ measures are generically stable}

We now prove that if $\mu$ is a \fim\ measure in a discrete theory, then the definable type $r_\mu$ constructed in the previous section is generically stable in the randomization. This will  follow immediately from the following \textit{nearly} identical characterizations of ``$\mu$ is \fim" and ``$r_\mu$ is generically stable".

\begin{proposition}\label{prop:ABC}
Let $T$ be a discrete $\cL$-theory. Suppose $\mu\in\kM_x(\cU)$ is definable.
\begin{enumerate}[$(a)$]
\item  $\mu$ is \fim\ if and only if for every $\cL$-formula $\varphi(x,y)$, 
  \[ 
  \lim_{n \to \infty} \textstyle \left(\sup_y\mathbb{E}[|\Av(\xbar)(\varphi(x,y))-F^\varphi_\mu(y)|]\right)^{r^{(n)}_\mu} = 0.
  \]
  \item $r_\mu$ is generically stable if and only if for every $\cL$-formula $\varphi(x,y)$, 
  \[
  \lim_{n\to\infty}\textstyle \left(\sup_y |\mathbb{E} [\Av(\xbar)(\varphi(x,y))-F^\varphi_\mu(y)]|\right)^{r^{(n)}_\mu}=0.
  \]
  \end{enumerate}
\end{proposition}
\begin{proof}
We first prove a claim.\medskip

\noindent\emph{Claim.} Let $\varphi(x,y)$ be an $\cL$-formula, and fix $n\geq 1$.
\begin{enumerate}[$(i)$]
\item  $\chi^{\Eb[\varphi]}_{r_\mu,n}(\xbar)=  \sup_y |\Eb[\Av(\bar{x})(\varphi(x,y)) - F^\varphi_\mu(y)]|$.
\item $\int\chi^\varphi_{\mu,n}(\xbar)\, d\mu^{(n)}=\left(\sup_{\bar{y}} \Eb[|\Av(\bar{x})(\varphi(x,\bar{y})) - F^\varphi_\mu(y)|]\right)^{r^{(n)}_\mu}$. 
\end{enumerate}
\noindent\emph{Proof.}
For part $(i)$, we have
\begin{align*}
  \chi^{\Eb[\varphi]}_{r_\mu,n}(\xbar) &=\textstyle  \sup_y  |\Av (\bar{x})(\Eb[\varphi(x,{y})]) - F^{\Eb[\varphi]}_{r_\mu}(y))|  \\
                                    & \overset{(1)}{=} \textstyle \sup_y |\Av (\bar{x})(\Eb[\varphi(x,{y})]) - \Eb[F^\varphi_\mu(y)]| \\
                                    & \overset{(2)}{=} \textstyle \sup_y |\Eb[\Av(\bar{x})(\varphi(x,y)) - F^\varphi_\mu(y)]|.
  \end{align*}
Equality $(1)$ uses  \cref{fact:unique-definable-transfer}. Equality $(2)$ uses linearity of $\Eb$.

For part $(ii)$, we have
\[
    \int \chi_{\mu,n}^{\varphi}(\bar{x})d\mu^{(n)} \overset{(1)}{=} \left(\Eb[\chi_{\mu,n}^{\varphi}(\bar{x})]\right)^{r_{\mu^{(n)}}} \overset{(2)}{=} \left({\textstyle \sup_{y}}\Eb[|\Av(\xbar)(\varphi(x,y) - F_{\mu}^{\varphi}(y)|]\right)^{r_{\mu}^{(n)}}. 
\]
Equality $(1)$ uses Corollary \ref{cor:rmu-def}. Equality $(2)$ uses Corollary \ref{cor:prod-same}, the definition of $\chi^\varphi_{\mu,n}$, and  Fact \ref{fact:randomizationsb}. \hfill $\square_{\text{claim}}$\medskip

Part $(a)$ of the proposition now follows immediately from part $(ii)$ of the claim and  Lemma \ref{lem:fim-char-cont}, which we can apply to $T$ viewed as a continuous theory in the usual way. 

For part $(b)$, first note that for any $\cL$-formula $\varphi(x,y)$, part $(i)$ of the claim yields
\[
\left({\textstyle \sup_y|\Eb[\Av(\xbar)(\varphi(x,y))-F^\varphi_\mu(y)]|}\right)^{r^{(n)}_\mu} = \left(\chi^{\Eb[\varphi]}_{r_\mu,n}(\xbar)\right)^{r^{(n)}_\mu} = \int \chi^{\Eb[\varphi]}_{r_\mu,n}(\xbar)\, dr^{(n)}_\mu(\xbar).
\]
By Lemma \ref{lem:fim-char-cont} applied to the continuous theory $T^R$ and the type $r_\mu$, we see that $r_\mu$ satisfies the right-hand condition in $(b)$ if and only if, for every $\cL$-formula $\varphi(x,y)$, $r_\mu$ is \fim\ relative to $\Eb[\varphi(x,y)]$. By Fact \ref{fact:fimgs-cont}$(b)$, it follows that $r_\mu$ satisfies the right-hand condition in $(b)$ if and only if, for every $\cL$-formula $\varphi(x,y)$, $r_\mu$ is generically stable relative to $\Eb[\varphi(x,y)]$. In particular, this immediately yields the left-to-right implication in $(b)$. 

Conversely, suppose $r_\mu$ satisfies the right-hand condition in $(b)$. Let $(a_i)_{i<\omega}$ be a Morley sequence in $r_\mu$. Given an $\cL$-formula $\varphi(x,y)$, by the previous paragraph, we have that $r_\mu$ is generically stable relative to $\Eb[\varphi(x,y)]$. So by Fact \ref{fact:fimgs-cont}$(b)$, we have that  $\lim_{i\to \infty} \Eb[\varphi(a_i,b)] = F^{\Eb[\varphi]}_{r_\mu}(b)$ for all $\cL$-formulas $\varphi(x,y)$ and all $b$. By quantifier elimination and continuity, this implies that $\lim_{i \to \infty} \psi(a_i,b) = F^{\psi}_{r_\mu}(b)$ for all $\cL^R$-formulas $\psi(x,y)$ and parameters $b$. So $r_\mu$ is generically stable.
\end{proof}

\begin{cor}\label{cor:fim-transfer} 
Let $T$ be a discrete theory. Suppose $\mu \in \mathfrak{M}_{x}(\mathcal{U})$  is \fim\ over $M$. Then  $r_{\mu} \in S_{x}(\mathcal{C})$ is generically stable over $M^{\Omega}$ (and thus also \fim\ over $M^{\Omega}$). 
\end{cor}
\begin{proof}
Since $\mu$ is \fim, the right-hand side of Proposition \ref{prop:ABC}$(a)$ holds. By Jensen's inequality, this implies that the right-hand side of Proposition \ref{prop:ABC}$(b)$ holds. Thus $r_\mu$ is generically stable. Since $r_\mu$ is invariant over $M^\Omega$ (see Remark \ref{remark:invariance}), it follows that $r_\mu$ is generically stable over $M^\Omega$.
\end{proof}

We can establish the converse of Corollary \ref{cor:fim-transfer}  in the case of types. 

\begin{prop}\label{prop:rpgs}
Suppose $p\in S_x(\cU)$ is a definable type such that $r_p$ is generically stable. Then $p$ is generically stable. 
\end{prop}
\begin{proof}
Let $p$ be definable over some $M \models T$. Suppose that $p$ is not generically stable. This implies that there is $a_{<\omega} \models p^{(\omega)}|_M$, $b$, and a formula $\varphi(x,y)$ such that $\lim_{i \to \infty} \varphi(a_i,b)$ does not exist.  

Let $N$ be a model containing $Ma_{<\omega}b$. Let $\Omega$ be an atomless probability algebra. For each $i<\omega$, let $a'_i$ be the element of $N^\Omega$ that is constantly equal to $a_i$. Likewise, let $b' \in N^\omega$ be the element constantly equal to $b$, and let $M' \subseteq N^\Omega$ be the set of elements $c'$ which are constantly equal to some $c \in M$. We clearly have that $\lim_{i \to \infty} \Eb[\varphi(a_i,b)]$ does not exist. So by Fact \ref{fact:fimgs-cont}$(b)$, to show that $r_p$ is not generically stable, it suffices to establish that $a'_{<\omega} \models r_p^{(\omega)}|_{M'}$. (Note that $r_p$ is automatically definable over $M'$.)

We know by \cref{cor:prod-same} that $r_p^{(\omega)}$ is the unique $M'$-definable type in $T^R$ corresponding to $p^{(\omega)}$. For any formula $\varphi(\bar{x},\bar{c})$ (in the language of $T$) with $\bar{c} \in M$, we have that $\Eb[\varphi(a'_{<\omega},\bar{c}')]$ is either $0$ or $1$. Furthermore, we have that each of the following is equivalent to the next:
    $\Eb[\varphi(a'_{<\omega},\bar{c}')] = 1$ $\Toot$ $\varphi(a_{<\omega},\bar{c})$ holds $\Toot$ 
    $F^{\varphi}_{p^{(\omega)}}(\bar{c}) = 1$ $\Toot$
    $\Eb[F^{\varphi}_{p^{(\omega)}}(\bar{c}')] = 1$ $\Toot$
    $F^{\Eb[\varphi]}_{r_p^{(\omega)}}(\bar{c}') = 1$. 
   Therefore, by quantifier elimination, we have that $a'_{<\omega} \models r_p^{(\omega)} | M'$. 
\end{proof}

\begin{remark}\label{rem:sa-random}
It is possible to characterize self-averaging measures (\cref{def:GSM}) in terms of randomizations. 
A definable measure $\mu$ in $\mathfrak{M}_{x}(\mathcal{U})$ is self-averaging if and only if $r_\mu$ is generically stable with respect to any formula with parameters that are constant functions (in the standard model of the randomization), i.e.,  for any $\cL$-formula $\varphi(x,y)$ and any $\e > 0$, there is an $n_{\varphi,\e}$ such that for any Morley sequence $(a_i)_{i< \omega}$ in $r_\mu$ and any $b \in \Uc$, 
\[
|\{i < \omega : |\Eb[\varphi(a_i,f_{b})] -(\Eb[\varphi(x,f_b)])^{r_{\mu}}| > \e\}|< n_{\varphi,\e}.
\] 
This characterization follows directly from saturation of the monster model. Since this property is entailed by generic stability of $r_\mu$, we have another proof of \cref{thm:fim-main}$(a)$: $\mu$ being \fim\ implies that $r_\mu$ is generically stable, which implies that $\mu$ is self-averaging. Note that if one could strengthen the above characterization to include formulas with arbitrary parameters, then this would yield the equivalence of self-averaging for $\mu$ and generic stability for $r_\mu$.
\end{remark}

\begin{remark}\label{rem:fim-transfer-NIP}
    If $T$ is NIP then Corollary \ref{cor:saNIP} and Remark \ref{rem:sa-random} yield the equivalence between our three main notions ($\mu$ is \fim,  $r_\mu$ is generically stable, $\mu$ is self-averaging). However, as noted by Will Johnson, one one can establish the equivalence between \fim\ for $\mu$ and generic stability for $r_\mu$ in the NIP case using only Proposition \ref{prop:prod-same-1} together with known results. In particular, assume $T$ is NIP. Then a measure $\mu\in\kM_x(\cU)$ is \fim\ if and only if it is self-commuting by \cite[Theorem 3.2]{HPS}. Moreover, since $T^R$ is also NIP  \cite{BY-NIP},  a type $p\in S_x(\mathcal{C})$ is generically stable if and only if it is self-commuting by \cite[Prop. 3.2]{HP} (adapted to continuous logic). Finally, Proposition \ref{prop:prod-same-1} shows that a definable measure $\mu$ is self-commuting if and only if $r_\mu$ is self-commuting. 
\end{remark}

\subsection{Randomization of \famm\ measures are \famm}\label{sec:fam}

In this section we show that the transfer map $\mu\mapsto r_\mu$ also preserves \famm. The proof will ultimately boil down to transferring \famm\ when $\mu$ is an average measure $\Av(\abar)$, which we can do using Corollary \ref{cor:fim-transfer} since $\mu$ is actually \fim\ in this case. It could be interesting to find an easier or more direct proof that $r_\mu$ is \famm\ when $\mu$ is an average measure. 

\begin{lemma} \label{lem:fam-transfer}
Suppose $\mu\in\kM(\cU)$ is definable over $M\prec\cU$.
Fix an $\cL$-formula $\varphi(x,y)$ and some $\epsilon > 0$. Assume $\abar = a_0,\ldots,a_n$ is a tuple from $M^{x}$ such that 
\begin{equation*} 
\sup_{b \in \mathcal{U}^{y}} |\mu(\varphi(x,b)) - \Av(\abar)(\varphi(x,b))| < \epsilon.
\end{equation*} 
Then 
\begin{equation*}
\sup_{g \in \mathcal{C}^{y}}|(\Eb[\varphi(x,g)])^{r_{\mu}} -  \Eb[\varphi(x,g)])^{r_{\Av(\abar)}}| \leq \epsilon.
\end{equation*} 
\end{lemma} 

\begin{proof} 
Since both $r_{\mu}$ and $r_{\Av(\bar{a})}$ are $\mathcal{U}^{\Omega}$-definable types, it suffices to show the inequality for any $h \in (\mathcal{U}^{\Omega}_0)^{y}$. Choose a finite measurable partition $\mathcal{A}$ of $\Omega$ such that each member of $h$ is constant on each element of $\mathcal{A}$. Then
\begin{align*}
\left( \Eb[\varphi(x,h)] \right)^{r_{\mu}} &= \sum_{A \in \mathcal{A}} \mathbb{P}(A)\mu(\varphi(x,h|_{A}))\\
&\approx_{\epsilon} \sum_{A \in \mathcal{A}} \mathbb{P}(A) \Av(\overline{a})(\varphi(x,h|_{A}))\\
&= \left( \Eb[\varphi(x,h)] \right)^{r_{\Av(\abar)}} \qedhere
\end{align*}

\end{proof} 

\begin{remark} 
We take a moment to make a subtle point which arises in the proof of the next theorem. If we let $\abar =a_1,\ldots,a_n$ be elements in $\mathcal{U}^{x}$, then the type $r_{\Av(\abar)}$ is almost never a realized type. The only exception is when $\abar$ consists of a single element. 
\end{remark}

\begin{theorem}\label{thm:fim-transfer} 
If $\mu \in \mathfrak{M}_{x}(\mathcal{U})$ is \famm\ over $M$, then $r_{\mu}$ is \famm\ over $M^{\Omega}$. 
\end{theorem} 

\begin{proof} 
First, note that $\mu$ is definable over $M$ and so $r_\mu$ is well-defined.
Suppose we are given $\mathcal{L}$-formulas $\psi_1(x,y),\ldots,\psi_n(x,y)$ and $\epsilon > 0$. We want to show that for any continuous connective $c\colon [0,1]^{n} \to [0,1]$, we have a \famm\ approximation for the formula $\Psi(x,y):=c(\mathbb{E}[\psi_1(x,y)],\ldots,\mathbb{E}[\psi_n(x,y)])$ (by quantifier elimination for randomizations, it is sufficient to check formulas of this form). Recall that $c$ is uniformly continuous and so for any $\epsilon > 0$ there exists some $\delta_{\epsilon}$ such that if $d(a_i,b_i) \leq \delta_{\epsilon}$ for each $i \leq n$, then $d(c(\abar),c(\bbar)) < \epsilon$. 

By a standard encoding argument, there exists $\abar = a_1,\ldots,a_n$ in $M^{x}$ such that for each $i \leq n$,
\begin{equation*} 
\sup_{b \in \mathcal{U}^{y}} |\mu(\varphi_i(x,b)) - \Av(\abar)(\varphi_i(x,b))| < \delta_{\frac{\epsilon}{3}}.\tag{$\dagger$}
\end{equation*} 
Since $\Av(\abar)$ is \fim\ over $M$, $r_{\Av(\abar)}$ is generically stable over $M^{\Omega}$ (by Corollary \ref{cor:fim-transfer}) and hence \famm\ over $M^{\Omega}$. Hence there exists $\fbar = f_1,\ldots,f_m$ from $M^{\Omega}$ such that 
\begin{equation*} 
\sup_{g \in \mathcal{C}^{y}} |(\Psi(x,g))^{r_{\Av(\abar)}} - \Av(\fbar)(\Psi(x,g))| < \frac{\epsilon}{3}\tag{$\dagger\dagger$}
\end{equation*} 
Therefore, for any $g \in \mathcal{C}^{y}$, we have  
\begin{align*} 
(\Psi(x,g))^{r_{\mu}} &= c\Big((\Eb[\varphi_{1}(x,g)])^{r_{\mu}},\ldots,(\Eb[\varphi_{n}(x,g)])^{r_{\mu}} \Big) \\
& \approx_{\frac{\epsilon}{3}} c\Big((\Eb[\varphi_{1}(x,g)])^{r_{\Av(\abar)}},\ldots,(\Eb[\varphi_{n}(x,g)])^{{r_{\Av(\abar)}}} \Big) \\
&= \Big(\Psi(x,g) \Big)^{r_{\Av(\abar)}} \\
&\approx_{\frac{\epsilon}{3}} \Av(\fbar)(\Psi(y,g)), 
\end{align*} 
where the first approximation follows from $(\dagger)$, Lemma \ref{lem:fam-transfer}, and choice of $\delta_{\frac{\epsilon}{3}}$, and the second approximation follows from $(\dagger\dagger)$.
Thus $r_{\mu}$ is \famm\ over $M^{\Omega}$. 
\end{proof}

\subsection{Randomization of \dfs\ measures need not be \dfs}\label{sec:dfs}

Here we will give an example showing that \emph{dfs} can fail to transfer through the map $\mu\mapsto r_\mu$. We work with the (discrete) theory $\Ti$ defined in Section 7 of \cite{CoGaHa}. Familiarity with this theory is assumed. 

\begin{proposition}
Let $\cU$ be a monster model of $\Ti$, and let $q\in S_Q(\cU)$ be the type defined in \cite[Cor.~7.12]{CoGaHa}. Then $q$ is $\emptyset$-definable and finitely satisfiable in every small model of $\Ti$, but $r_q$ is not finitely satisfiable in any small model of $(\Ti)^R$. 
\end{proposition}
\begin{proof}
    It is established in \cite[Cor.~7.12]{CoGaHa} that $q$ is $\emptyset$-definable and finitely satisfiable in every small model of $\Ti$. Let $N$ be a small model of $(T_{\nicefrac{1}{2}}^\infty)^R$ and, toward a contradiction, assume that $r_q$ is finitely satisfiable in $N$.

  By \cite[Prop.~2.1.10]{AnGoKe}, we may assume that there is a model $M$ of $T_{\nicefrac{1}{2}}^\infty$ and an atomless probability space $(\Omega,\mathcal{B},\nu)$ such that $N$ is isomorphic to (the metric reduction of) some collection of functions from $\Omega$ to $M$ that are compatible with the measure on $\Omega$ in the appropriate way.

 Let $\mu$ be the measure defined in \cite[Lem.~7.13]{CoGaHa}, and  
  let $a$ be a realization of $r_\mu|_M$. By definition of $q$, $r_q$ contains the closed conditions $\Eb[a \sqin y] = 1$ and $\Eb[\ell(y) = \frac{1}{2}] = 1$. This means that there must be some $b \in M$ such that $\Eb[a\sqin b] > \frac{3}{4}$ and $\Eb[\ell(b) = \frac{1}{2}] > \frac{3}{4}$.

  By the definition of $\mu(x)$, we have that for each $c \in Q^M$,
  \[
    \Eb[a \sqin c] = \int_{ \Omega} \textnormal{st}(\ell(c(\omega)))\, d\nu,
  \]
  where $\textnormal{st}$ is the standard part map taking elements of $[0,1]^M$ to $[0,1] \subset \mathbb{R}$. Since $\Eb[\ell(b) = \frac{1}{2}] > \frac{3}{4}$, we must have that
  \[
    \frac{3}{4} < \Eb[a \sqin b] = \int_{ \Omega} \textnormal{st}(\ell(b(\omega)))\, d\nu < \frac{1}{2}\cdot\frac{3}{4} + \frac{1}{4} < \frac{3}{4},
  \]
  which is absurd.
\end{proof}

\section{Generically stable types and NTP$_2$}\label{sec:NTP2}

As stated in the introduction, the following question remains open for an arbitrary theory (see \cite[Section 8.1]{CoGaHa} for further related discussion).

\begin{question}
    Is the Morley product of two generically stable types  generically stable?
\end{question}

The goal of this section is to give a combinatorial characterization of theories in which the answer to this question is positive. As an application, we will  show that generically stable types are closed under Morley products in any NTP$_2$ theory. Throughout this section, $T$ is a complete first-order theory with monster model $\cU$. The notion of an ``\textit{ict}-pattern" comes from work of Shelah \cite{Sh783} (see also \cite[Fact 3.7]{ChNTP2}).  We recall the following special case.

\begin{definition}\label{def:urp}
A \textbf{uniform \textit{ict}-pattern (in $T$)} consists of a formula $\varphi(x,y)$ and an array $\{a_{i,j}:i,j<\omega\}$ from $\Uc^x$ such that for any function $f\colon\omega\to\omega$, the set $\{\varphi(a_{i,j},y)^{f(i)=j}:i,j<\omega\}$ is consistent.
\end{definition}

The next fact is a standard exercise.

\begin{fact}\label{fact:NIPrp}
$T$ is \textnormal{NIP} if and only if it does not admit a uniform ict-pattern.
\end{fact}

\begin{remark}\label{rem:bij}
The existence of uniform \textit{ict}-patterns would not be affected by restricting to bijections in Definition \ref{def:urp}. Specifically, suppose  there is a formula $\varphi(x,y)$ and an array $\{a_{i,j}:i,j<\omega\}$ in $\Uc^x$ such that for any bijection $f\colon \omega\to\omega$, $\{\varphi(a_{i,j},y)^{f(i)=j}:i,j<\omega\}$ is consistent. Then we claim that $T$ admits a uniform \textit{ict}-pattern using the same formula $\varphi(x,y)$. Indeed, by compactness we may extend the array to be indexed $\{a_{i,j}:i,j<\omega^2\}$, while maintaining the same conclusion for any bijection $f\colon\omega^2\to\omega^2$. Then $\varphi(x,y)$ yields a uniform \textit{ict}-pattern using the sub-array $\{a_{i,\omega\cdot i+j}:i,j<\omega\}$. In particular, given an arbitrary function $g\colon \omega\to\omega$, one can construct a bijection $f\colon \omega^2\to\omega^2$ with the property that $f(i)=\omega\cdot i+g(i)$ for all $i<\omega$.
\end{remark}

We now introduce a new kind of \textit{ict}-pattern, which is built using stable sequences (recall Definition \ref{def:stabseq}) and involves a tri-partitioned formula. 

\begin{definition}\label{def:gss}
A \textbf{stable \textit{ict}-pattern (in $T$)} consists of a formula $\varphi(x,y,z)$ and stable sequences $(a_i)_{i<\omega}$ from $\Uc^x$ and $(b_i)_{i<\omega}$ from $\Uc^y$ such that for any bijection $f\colon\omega\to\omega$, the set $\{\varphi(a_i,b_j,z)^{f(i)=j}:i,j<\omega\}$ is consistent.
\end{definition}

\begin{remark}$~$
\begin{enumerate}[$(a)$]
\item If $T$ admits a stable \textit{ict}-pattern involving the formula $\varphi(x,y,z)$ then,  by Remark \ref{rem:bij}, it admits a uniform \textit{ict}-pattern involving $\varphi(x,y;z)$.
\item If one replaces ``bijection" with ``function" in Definition \ref{def:gss}, then the resulting pattern is impossible. For example, if $(a_i)_{i<\omega}$ from $\Uc^x$ is stable, $(b_i)_{i<\omega}$ from $\Uc^y$ is arbitrary, and $f\colon\omega\to\omega$ has an infinite/co-infinite fiber, then for any $\varphi(x,y,z)$, the set  $\{\varphi(a_i,b_j,z)^{f(i)=j}:i,j<\omega\}$ is inconsistent. 
\end{enumerate}
\end{remark}

The following  straightforward observations concerning Morley products will be needed for the main result below.

\begin{proposition}\label{prop:sec4tech}
    Fix an integer $\ell\geq 1$.
    \begin{enumerate}[$(a)$]
\item For each $1\leq t\leq \ell$, let $(a^t_i)_{i<\omega}$ be a sequence from $\Uc^{x_t}$ such that $p_t:=\Avtp((a^t_i)_{i<\omega})$ is a  complete type in $S_{x_t}(\Uc)$. Then an  $\Lc_{\Uc}$-formula $\theta(x_1,\ldots,x_\ell)$ is in $p_1\otimes\ldots\otimes p_\ell$ if and only if 
\[
\forall^\infty i_1,\ldots,\forall^\infty i_\ell,~ \Uc\models \theta(a^1_{i_1},\ldots,a^\ell_{i_\ell})
\]
(where $\forall^\infty$ denotes ``for all but finitely many").
\item For each $1\leq t\leq \ell$, let $p_t\in S_{x_t}(\Uc)$ be an $M$-invariant type. Assume that any two (not necessarily distinct) types chosen from $\{p_1,\ldots,p_\ell\}$ commute. Then for any Morley sequence $(a_{1,i}\ldots a_{\ell,i})_{i<\omega}$ in $p_1\otimes\ldots\otimes p_\ell$ over $M$, the sets $\{a_{t,i}:i<\omega\}$ for $1\leq t\leq \ell$ are mutually indiscernible over $M$. 
    \end{enumerate}
\end{proposition}
\begin{proof}
Part $(a)$ is left to the reader (use induction on $\ell$). For part $(b)$, the main claim is that given integers $n_1,\ldots,n_\ell\geq 0$ and, for each $1\leq t\leq\ell$,   pairwise distinct indices $i^t_1,\ldots,i^t_{n_t}<\omega$, we have
\[
\textstyle (a_{1,i^1_1},\ldots, a_{1,i^1_{n_1}},\ldots ,a_{\ell,i^\ell_1},\ldots,a_{\ell,i^\ell_{n_\ell}})\models \bigotimes_{t=1}^\ell\bigotimes_{j=1}^{n_t}p_t(x_{t,j})|_M.
\]
This yields the desired result  since the above type depends only on $n_1,\ldots,n_\ell$. We leave the proof of the above claim  to the reader (use induction on $n_1+\ldots+n_\ell$). 
\end{proof}

We now prove the main result of this section.

\begin{theorem}\label{thm:gsrp}
$T$ admits a stable \textit{ict}-pattern if and only if there are generically stable types $p\in S_x(\Uc)$ and $q\in S_y(\Uc)$ such that $p\otimes q$ is not generically stable.
\end{theorem}
\begin{proof}
Suppose first that $(\varphi(x,y,z),(a_i)_{i<\omega},(b_i)_{i<\omega})$ is a stable \textit{ict}-pattern in $T$. By Fact \ref{fact:gen-stab-seq-and-types}$(a)$, the types $p=\Avtp(a_{<\omega})\in S_x(\Uc)$ and $q=\Avtp(b_{<\omega})\in S_y(\Uc)$ are well-defined and  generically stable. We will show that $p\otimes q$ is not generically stable.

Fix $M\prec\Uc$ such that $p$ and $q$ are generically stable over $M$.
View $(p\otimes q)^{(\omega)}|_M$ as a type in variables $x_{<\omega}y_{<\omega}$, and consider  the set
\[
\Gamma=(p\otimes q)^{(\omega)}|_M\cup\{\varphi(x_i,y_j,z)^{i=j}:i,j<\omega\}.
\]
We show that $\Gamma$ is finitely satisfiable. Let $\Gamma_0\subseteq\Gamma$ be finite. Choose $n$ such that $\Gamma_0$ only involves $x_i$ and $y_i$ for $i<n$. Without loss of generality, we can assume $\Gamma_0$ is of the form $\theta(x_0,\ldots,x_{n-1},y_0,\ldots,y_{n-1})\wedge\bigwedge_{i,j<n}\varphi(x_i,y_j,z)^{i=j}$ where $\theta(\bar{x},\bar{y})\in (p\otimes q)^{(n)}|_M$. By Proposition \ref{prop:sec4tech}$(a)$, we can choose pairwise distinct $s_1,\ldots,s_n<\omega$ and $t_1,\ldots,t_n<\omega$  such that $\theta(a_{s_1},\ldots,a_{s_n},b_{t_1},\ldots,b_{t_n})$ holds.  Let $f\colon \omega\to\omega$ be a bijection such that $f(s_i)=t_i$ for all $1\leq i\leq n$. By assumption, there is some $c$ such that for all $s,t<\omega$, $\varphi(a_s,b_t,c)$ holds if and only if $f(s)=t$. So for $i,j<n$, we have $\varphi(a_{s_i},b_{t_j},c)$ if and only if $i=j$.  
Therefore $(a_{s_1},\ldots,a_{s_n},b_{t_1},\ldots,b_{t_n},c)$ realizes $\Gamma_0$. 

Let $(a^*_i,b^*_i,c)_{i<\omega}$ be a realization of $\Gamma$. Then $(a^*_i,b^*_i)_{i<\omega}$ is a Morley sequence in $p\otimes q$ over $M$, and $\varphi(a^*_i,b^*_i,c)$ holds for all $i<\omega$. So to show that $p\otimes q$ is not generically stable, it suffices to show that $\varphi(x,y,c)\not\in p\otimes q$. To see this, note first that $(a^*_i)_{i<\omega}$ and $(b^*_i)_{i<\omega}$ are Morley sequences in $p$ and $q$ (respectively) over $M$, and so $p=\Avtp((a^*_i)_{i<\omega})$ and $q=\Avtp((b^*_i)_{i<\omega})$ by generic stability. Moreover, for any fixed $i<\omega$, we have $\Uc\models\neg\varphi(a^*_i,b^*_j,c)$ for all $j\neq i$. So $\neg\varphi(x,y,c)\in p\otimes q$ by Proposition \ref{prop:sec4tech}$(a)$. 
\medskip

Conversely, suppose there are generically stable types $p\in S_x(\Uc)$ and $q\in S_y(\Uc)$ such that $p\otimes q$ is not generically stable. Then we have some $\Lc_{\Uc}$-formula $\varphi(x,y,c)$ and a Morley sequence $(a_i,b_i)_{i<\omega}$ in $p\otimes q$ over $M$ such that $\varphi(x,y,c)\not\in p\otimes q$ but $\varphi(a_i,b_i,c)$ holds for all $i<\omega$. We will show that, after possibly passing to subsequences, $(\varphi(x,y,z),(a_i)_{i<\omega}, (b_i)_{i<\omega})$ is a  stable \textit{ict}-pattern. Note first that $(a_i)_{i<\omega}$ and $(b_i)_{i<\omega}$ are Morley sequences in $p$ and $q$ (respectively) over $M$. So $p=\Avtp((a_i)_{i<\omega})$ and $q=\Avtp((b_i)_{i<\omega})$ by generic stability. 

 Since $\varphi(x,y,c)\not\in p\otimes q$, and $p\otimes q=q\otimes p$,  it follows from Proposition \ref{prop:sec4tech}$(a)$ that there is a cofinite set $I\subseteq\omega$ such that for all $t\in I$, the sets $\{i<\omega:\varphi(a_i,b_t,c)\}$ and $\{i<\omega:\varphi(a_t,b_i,c)\}$ are finite. We inductively construct an increasing subsequence $(i_k)_{k<\omega}$ from $I$ such that $\neg\varphi(a_{i_k},b_{i_\ell},c)$ holds for all $k\neq \ell$. Let $i_0\in I$ be arbitrary. Given $i_0<\ldots<i_k<\omega$, choose $i_{k+1}\in I$ such that $i_{k+1}>i_k$ and 
 \[
 i_{k+1}\not\in \bigcup_{t=1}^k\{i<\omega:\varphi(a_i,b_{i_t},c)\}\cup\{i<\omega:\varphi(a_{i_t},b_i,c)\}.
 \]

After passing to the subsequence  above, we may assume that $\varphi(a_i,b_j,c)$ holds if and only if $i=j$. We now verify that $(\varphi(x,y,z),(a_i)_{i<\omega},(b_i)_{i<\omega})$ is a  stable \textit{ict}-pattern. Note that $(a_i)_{i<\omega}$ and $(b_i)_{i<\omega}$ are Morley sequences in $p$ and $q$ (respectively) over $M$, and thus are stable. So fix a bijection $f\colon\omega\to\omega$. Then by  Proposition \ref{prop:sec4tech}$(b)$ (with $\ell=2$), we have
\[
a_0a_1a_2\ldots b_0b_1b_2\ldots \equiv a_{0}a_{1}a_{2}\ldots b_{f(0)}b_{f(1)}b_{f(2)}\ldots.
\]
Therefore $\{\varphi(a_i,b_j,z)^{f(i)=j}:i,j<\omega\}$ is realized by some conjugate of $c$.
\end{proof}

\begin{remark}
The proof shows that if $T$ admits a stable \textit{ict}-pattern, then it admits one of the form $(\varphi(x,y,z),(a_i)_{i<\omega},(b_i)_{i<\omega})$, where $(a_i,b_i)_{i<\omega}$ is a Morley sequence in $\Avtp((a_i)_{i<\omega})\otimes \Avtp((b_i)_{i<\omega})$ over some small model $M$. In particular, the sequences are mutually indiscernible sets. More precisely, we see that $T$ admits a  stable \textit{ict}-pattern if and only if there are types $p\in S_x(\Uc)$ and $q\in S_y(\Uc)$, which are generically stable over some $M\prec\Uc$, a Morley sequence $(a_i,b_i)_{i<\omega}$ in $p\otimes q$ over $M$, and formula $\varphi(x,y,z)$ such that $\{\varphi(a_i,b_j,z)^{i=j}:i,j<\omega\}$ is consistent. 
\end{remark}

Once again, there is no known example of a theory  that actually admits a stable \emph{ict}-pattern. The last goal of this section is to show that if such an example exists, then it must have TP$_2$.

\begin{definition}\label{def:TP2}
A formula $\varphi(x,y)$ has \textbf{TP}$_2$ if there is an array $\{a_{i,j}:i,j<\omega\}$ in $\Uc^y$ satisfying the following properties:
\begin{enumerate}[$(i)$]
\item For all functions $f\colon\omega\to\omega$, $\{\varphi(x,a_{i,f(i)}):i<\omega\}$ is consistent.
\item There is some $k\geq 2$ such that for all $i<\omega$, $\{\varphi(x,a_{i,j}):j<\omega\}$ is $k$-inconsistent.
\end{enumerate}
$T$ is \textbf{NTP}$_2$ if there is no formula with TP$_2$.
\end{definition}

\begin{remark}\label{rem:inj}
Using a similar compactness argument as in Remark \ref{rem:bij}, one can show that in condition $(i)$ of the definition of TP$_2$, it suffices to only consider bijections $f\colon\omega\to\omega$.
\end{remark}

\begin{proposition}\label{prop:NTP2ict}
If $T$ is $\textnormal{NTP}_2$ then it does not admit a  stable \textit{ict}-pattern. 
\end{proposition}
\begin{proof}
Suppose $T$ admits a  stable \textit{ict}-pattern $(\varphi(x,y,z),(a_i)_{i<\omega},(b_i)_{i<\omega})$. Set $q=\Avtp((b_i)_{i<\omega})$, and note that $q$ is generically stable by Fact \ref{fact:gen-stab-seq-and-types}$(a)$. In particular, $q$ is definable, so we may choose a formula $\psi(x,z)$ defining $q$ with respect to $\varphi(x,y,z)$. Let $\theta(z;x,y)$  be the formula $\varphi(x,y,z)\wedge \neg \psi(x,z)$. We will show that $\theta(z;x,y)$ has TP$_2$ witnessed by the array $\{(a_i,b_j):i,j<\omega\}$. 

With Remark \ref{rem:inj} in mind, first fix a bijection $f\colon \omega\to\omega$. Then we have $c\in\Uc^z$ such that $\varphi(a_i,b_j,c)$ holds if and only if $f(i)=j$. We want to show that for all $i<\omega$, $\theta(c;a_i,b_{f(i)})$ holds. So, in particular, we need to show that for all $i<\omega$, $\varphi(a_i,y,c)\not\in q$. But this follows from the fact that for any $i<\omega$, we have $\neg \varphi(a_i,b_j,c)$ for any $j\neq f(i)$. 

Now, by assumption there is some $k<\omega$ such that for any $a'c'\in\Uc^{xz}$, if $|\{j<\omega:\varphi(a',b_j,c')\}|\geq k$ then $\varphi(a',y,c')\in q$.  It follows that for any $i<\omega$, the set $\{\theta(z;a_i,b_j):j<\omega\}$ is $k$-inconsistent. 
\end{proof}

\begin{corollary}\label{cor:NTP2gs}
If $T$ is $\textnormal{NTP}_2$ then the Morley product of two generically stable types is generically stable.
\end{corollary}

In \cite{CoGa}, the first two authors provided a characterization of generically stable types in NTP$_2$ theories in terms of weak stationarity and a forking symmetry condition similar to the notion of ``generic simplicity" (defined by Simon in \cite{SimGSG}). Using various lemmas from \cite{SimGSG}, many of which are adapted from \cite{ChNTP2}, one can give an alternate proof of Corollary \ref{cor:NTP2gs} by showing that if $T$ is NTP$_2$ then this characterization of generic stability for types is closed under Morley products. 

\begin{remark}
It follows from the proof of Proposition \ref{prop:NTP2ict} that TP$_2$ can be obtained from a weaker version of a stable \textit{ict}-pattern in which only one of the two sequences is assumed to be stable. However, stability of at least one of the sequences is necessary to obtain TP$_2$. For example, if $T$ is the theory of the generic $3$-uniform hypergraph, then $T$ is simple (hence NTP$_2$), but one can easily find sequences  $(a_i)_{i<\omega}$ and $(b_i)_{i<\omega}$ such that for \emph{any} $I\seq \omega\times\omega$, the set $\{R(a_i,b_j,z)^{(i,j)\in I}:i,j<\omega\}$ is consistent. In other words, if one removes the assumption in Definition \ref{def:gss} that the sequences are stable, then the resulting ``pattern" can be obtained in any theory that is not $2$-dependent; but this pattern is \emph{a priori} much weaker than what is required to witness $2$-independence. Altogether, while we currently have no example of any theory admitting a stable \textit{ict}-pattern, this discussion suggests that the resolution  for $2$-dependent theories may be more accessible. 
\end{remark}

\section{Tame extensions for local measures}\label{sec:local}

Let $T$ be a complete theory with monster model $\cU$. Recall that a global measure $\mu\in\kM_x(\cU)$ is  \emph{smooth} if there is some $A\subset\cU$ such that $\mu$ is the unique global extension of $\mu|_A$. The following result is due to Keisler (see \cite[Theorem 3.16]{Keis}).

\begin{fact}[Keisler]\label{fact:Keisler}
A countable theory $T$  is NIP if and only if  any Keisler measure over a small model $M\prec\cU$ has a smooth global extension. 
\end{fact}

There are several important remarks to make here, some of which motivated the work in this section. First,  Keisler works with a weaker notion of smoothness and thus the left-to-right direction of the above fact is formally stronger that what is actually proved in \cite{Keis}. However, it was observed by Hrushovski, Pillay, and Simon \cite{HPS} that the mechanics of Keisler's proof work for what has now become the modern notion of smoothness (see \cite[Proposition 7.9]{Sibook} for a detailed proof). We also note that this direction does not require countability of $T$. 

Curiously enough, the right-to-left direction of Fact \ref{fact:Keisler} had been passed around as an open problem by a number of people in the field. It was only after completing the work in this section that the authors double-checked Keisler's paper and found the result stated there. Morever, since Keisler's notion of smoothness is weaker, he actually proves a \emph{stronger} result in this direction. On the other hand, there is a blanket countability assumption made throughout Keisler's paper, and he asks in \cite[Section 3.18]{Keis} to what extent this assumption is necessary. We do not claim to answer this question, since the overall picture of Keisler's work is rather  elaborate. However, we will prove a local analogue of Fact \ref{fact:Keisler}, which makes the cardinality of $T$ irrelevant. Since a local analogue of the left-to-right direction is a straightforward adaptation of the global argument, and with the faux open question above in mind, our main focus will be on the right-to-left direction. In addition to providing a local analogue for an NIP formula, our proof will also simplify several steps by drawing a more direction connection to VC-dimension and, in particular, Haussler's packing lemma and the Sauer-Shelah Lemma. In particular, these tools yield a quantitative argument for bounding the size of an $\epsilon$-separated family with respect to measures in an NIP context (see the discussion at the start of Subsection \ref{sec:smoothext} for more details). We also note that our result strenghtens Keisler's in the sense that we actually characterize NIP via the existence of ``strongly definable" extensions for measures, which is a signicantly weaker notion than smoothness.

\subsection{Preliminaries}

Throughout this section, we work with a fixed formula $\varphi(x,y)$. Recall that a \emph{$\varphi$-formula} is a finite Boolean combination of instances $\varphi(x,b)$ of $\varphi(x,y)$, for $b\in\cU^y$. By a \emph{$\varphi$-generated formula} we mean a (parameter-free) formula $\theta(x,\ybar)$ obtained as a Boolean combination of $\varphi(x,y_i)$ where $y_i$ is in the sort of $y$. In particular, $\varphi$-formulas correspond to \emph{instances} of $\varphi$-generated formulas.

Given $A\seq\cU$, we let $\kM_\varphi(A)$ be the space of finitely additive probability measures on the Boolean algebra of $\varphi$-formulas over $A$. We refer to a measure in $\kM_\varphi(A)$ as a \emph{$\varphi$-measure over $A$}.

\begin{definition}\label{defn:local-invariance}
Suppose $\mu\in\kM_\varphi(\cU)$ is a $\varphi$-measure, and $\theta(x,\ybar)$ is a $\varphi$-generated formula. Fix a set $A\subset\cU$. 
\begin{enumerate}[$(1)$]
\item $\mu$ is \textbf{$\theta$-invariant} over $A$ if for all $\bbar\in\cU^{\ybar}$, $\mu(\theta(x,\bbar))$ depends only on $\tp(\bbar/A)$. In this case, let $F^\theta_{\mu,A}\colon S_{\ybar}(A)\to [0,1]$ denote the resulting map.
\item $\mu$ is \textbf{$\theta$-definable} over $A$ if it is $\theta$-invariant over $A$ and $F^\theta_{\mu,A}$ is continuous.
\item $\mu$ is \textbf{strongly $\theta$-invariant} over $A$ if for all $\bbar\in\cU^{\ybar}$, $\mu(\theta(x,\bbar))$ depends only on $\tp_{\theta^*}(\bbar/A)$. In this case, let $F^{\theta,s}_{\mu,A}\colon S_{\theta^*}(A)\to [0,1]$ denote the resulting map. 
\item $\mu$ is \textbf{strongly $\theta$-definable} over $A$ if it is strongly $\theta$-invariant over $A$ and $F^{\theta,s}_{\mu,A}$ is continuous.
\item $\mu$ is \textbf{finitely $\theta$-approximated in $A$} if for all $\epsilon>0$ there is some $\abar\in A^{<\omega}$ such that for all $\bbar\in \cU^{\ybar}$, $\mu(\theta(x,\bbar))\approx_\epsilon \Av(\abar)(\theta(x,\bbar))$. 
\end{enumerate}
\end{definition}

We now list the basic implications between the above notions.

\begin{proposition}\label{prop:imps}
Suppose $\mu\in\kM_\varphi(\cU)$ is a $\varphi$-measure, and $\theta(x,\ybar)$ is a $\varphi$-generated formula. Fix a set $A\subset\cU$. 
\begin{enumerate}[$(a)$]
\item If $\mu$ is strongly $\theta$-invariant over $A$, then $\mu$ is $\theta$-invariant over $A$.
\item If $\mu$ is strongly $\theta$-definable over $A$, then $\mu$ is $\theta$-definable over $A$.
\item If $\mu$ is $\theta$-definable over $A$ and strongly $\theta$-invariant over $A$, then $\mu$ is strongly $\theta$-definable over $A$.
\item If $\mu$ is finitely $\theta$-approximated in $A$, then $\mu$ is strongly $\theta$-definable over $A$.
\end{enumerate}
\end{proposition}
\begin{proof}
Part $(a)$ is clear.

Part $(b)$. Let $\rho\colon S_{\ybar}(A)\to S_{\theta^*}(A)$ denote the restriction map. Recall that $\rho$ is continuous. Assume $\mu$ is strongly $\theta$-definable over $A$. Then $F^\theta_{\mu,A}$ exists (by part $(a)$) and is equal to $F^{\theta,s}_{\mu,A}\circ\rho$, and hence is continuous.

Part $(c)$. Let $\rho$ be as above. Note that $\rho$ is a continuous surjective function between compact Hausdorff spaces, and hence is a quotient map. Now assume $\mu$ is $\theta$-definable over $A$ and strongly $\theta$-invariant over $A$. Then $F^\theta_{\mu,A}$ is continuous and $F^{\theta,s}_{\mu,A}$ exists. Since $F^\theta_{\mu,A}=F^{\theta,s}_{\mu,A}\circ\rho$, it follows that $F^{\theta,s}_{\mu,A}$ is continuous by the universal property of quotient maps. 

Part $(d)$. Assume $\mu$ is finitely $\theta$-approximated over $A$. Note that if $\tp_{\theta^*}(\bbar/A)=\tp_{\theta^*}(\bbar'/A)$ then $\Av(\abar)(\theta(x,\bbar))=\Av(\abar)(\theta(x,\bbar'))$ for any $\abar\in A^{<\omega}$. It follows that $\mu$ is strongly $\theta$-invariant over $A$. To prove that $\mu$ is strongly $\theta$-definable over $A$, we will show that $F^{\theta,s}_{\mu,A}$ is a uniform limit of continuous functions, and hence is continuous. Fix some $\epsilon>0$. Since $\mu$ is finitely $\theta$-approximated there is some $\abar\in A^{<\omega}$ such that 
\[
\sup_{q \in S_{y}(A)}| F^{\theta,s}_{\mu,A}(q)- F^{\theta,s}_{\Av(\abar),A}(q)| = \sup_{b \in \mathcal{U}^{y}}|\mu(\theta(x,b)) - \Av(\overline{a})(\theta(x,b))|<\epsilon.
\]
Moreover, $F^{\theta,s}_{\Av(\abar),A}=\frac{1}{n}\sum_{i=1}^n \boldsymbol{1}_{\theta(a_i,\ybar)}$, and so $F^{\theta_s}_{\Av(\abar),A}$ is a uniform limit of continuous (hence continuous).  
\end{proof}

Next we define the obvious local formulation of smoothness. It is worth mentioning that this notion has been previously considered by many people in the field, although not in any published work as far as we know. Thus we take the opportunity  to state the definition and list some basic  facts. 

\begin{definition}
A measure $\mu\in\kM_\varphi(\cU)$ is \textbf{smooth over $A\subset\cU$} if $\mu$ is the unique measure in $\kM_\varphi(\cU)$ extending $\mu|_A\in \kM_\varphi(A)$.
\end{definition}

\begin{remark}\label{rem:nosmooth}
It is worth noting that, unlike the ``non-local" situation, a realized local type need not be smooth. In fact, there may be no smooth local measures at all. For example, let $T$ be the theory of the random graph and let $\varphi(x,y)$ be the edge relation. Fix $\mu\in\kM_\varphi(\cU)$ and $A\subset\cU$. Then for any $b\in\cU\backslash A$ and any $r\in [0,1]$, there exists some $\nu_r\in \kM_\varphi(\cU)$ such that $\nu_r|_A=\mu|_A$ and $\nu_r(\varphi(x,b))=r$ (e.g., this follows from \cite{LosMar}, or see \cite[Theorem 3.7]{StarBour}; we also note that if $\mu$ is a type, and $r\in\{0,1\}$, one can further assume that $\nu_r$ is a type). So, in particular, $\mu$ is not smooth over $A$.
\end{remark}

For ``non-local" measures, the notion of smoothness has a useful syntactic characterization, which also holds in the local case.

\begin{proposition}\label{prop:smoothchar}
Fix $\mu\in\kM_\varphi(\cU)$ and $A\subset\cU$.  Then $\mu$ is smooth over $A$ if and only if for any $\varphi$-generated formula $\theta(x,\ybar)$ and any $\epsilon>0$, there are finitely many $\cL$-formulas $\psi_{\epsilon,1}(\ybar),\ldots \psi_{\epsilon,n}(\ybar)$ over $A$, and $\varphi$-formulas 
\[
\chi^0_{\epsilon,1}(x),\ldots,\chi^0_{\epsilon,n}(x),\chi^1_{\epsilon,1}(x),\ldots,\chi^1_{\epsilon,n}(x)
\]
over $A$, which satisfy the following properties.
\begin{enumerate}[$(i)$]
\item $\psi_{\epsilon,1}(\ybar),\ldots,\psi_{\epsilon,n}(\ybar)$ partition $\cU^{\ybar}$.
\item For all $1\leq i\leq n$, if $\bbar\models\psi_{\epsilon,i}(\ybar)$ then  $\chi^0_{\epsilon,i}(\cU)\seq\theta(\cU,\bbar)\seq\chi^1_{\epsilon,i}(\cU)$.
\item For all $1\leq i\leq n$, $\mu(\chi^1_{\epsilon,i}(x))-\mu(\chi^0_{\epsilon,i}(x))\leq \epsilon$.
\end{enumerate}
\end{proposition}
\begin{proof}
The argument can be directly adapted  from the proof of the global result (first formulated in \cite[Lemma 2.3]{HPS}). See also \cite[Lemma 7.8]{Sibook} and \cite[Proposition 3.14]{StarBour} for other detailed proofs. The local analogue follows using the  same compactness argument applied to a classical result of {\L}o\'{s} and Marczewski \cite{LosMar} on extending measures (which also follows from a slightly earlier paper of Horn and Tarski \cite{HoTa}). See  \cite[Theorem 3.7]{StarBour} for a formulation suitable to the present context. 
\end{proof}

\begin{remark}
The proof of Proposition \ref{prop:smoothchar} does \emph{not}  yield that the formulas $\psi_{\epsilon,i}(\ybar)$ are $\theta^*$-formulas over $A$. However, the statement of the proposition clearly implies that we can take each $\psi_{\epsilon,i}(\ybar)$ to be a finite Boolean combination of the formulas 
\[
\forall x((\chi^0_{\epsilon,t}(x)\rightarrow\theta(x,\ybar))\wedge(\theta(x,\ybar)\rightarrow\chi^1_{\epsilon,t}(x)))
\]
where $1\leq t\leq n$. (If one only asks for a covering of $\cU^{\ybar}$ in condition $(i)$, rather than a partition, then one can take $\psi_{\epsilon_,i}(\ybar)$ to be exactly the formula above with $t=i$). In particular, each $\psi_{\epsilon,i}(\ybar)$ is a formula in the reduct of $\cL$ to $\varphi(x,y)$. 
\end{remark}

Finally, we connect smoothness to the rest of the notions listed at the start of this subsection. First we show that if a local measure is smooth over some model $M$ then it is finitely approximated in $M$. We will actually prove a sharper result that deals with finite approximations for arbitrary Borel sets. Altogether, this result is a local analogue of \cite[Proposition 7.10]{Sibook}. The overal strategy of the argument is roughly the same. However, we will remove the use of the weak law of large numbers, and instead give a slightly more elementary proof using atoms of finite Boolean algebras.

\begin{definition}
    Given $B\seq A\seq\cU$, let $\rho^\varphi_{B,A}\colon S_\varphi(B)\to S_\varphi(A)$ denote the restriction map. We write $\rho^\varphi_A$ for $\rho^{\varphi}_{\cU,A}$.
\end{definition}

\begin{proposition}\label{prop:smoothfam}
Suppose $\mu\in\kM_\varphi(\cU)$ is smooth over $A\subset\cU$. Fix finitely many Borel sets $X_1,\ldots,X_k\seq S_\varphi(A)$ and  $\varphi$-generated formulas $\theta_1(x,\ybar_1),\ldots,\theta_k(x,\ybar_k)$. Then for any $\epsilon>0$ there is some $\bar{a}\in (\cU^x)^{<\omega}$ such that for any $1\leq s,t\leq k$ and $\bbar\in\cU^{\ybar_s}$,
\[
\mu(\theta_s(x,\bbar)\cap X'_t)\approx_\epsilon\Av(\abar)(\theta_s(x,\bbar)\cap X'_t),
\]
where $X'_t=(\rho^\varphi_A)^{\text{-}1}(X_t)$.

Moreover, if $A$ is a model $M\prec\cU$, and each $X_t$ is clopen, then we may assume $\bar{a}\in(M^x)^{<\omega}$. In particular, $\mu$ is finitely $\theta$-approximated in $M$ for any $\varphi$-generated formula $\theta(x,\ybar)$.
\end{proposition}
\begin{proof}
Fix $\epsilon>0$. For each $1\leq s\leq t$, choose $\cL$-formulas $\psi_{s,1}(\ybar),\ldots \psi_{s,n_s}(\ybar)$ over $A$, and $\varphi$-formulas $\chi^0_{s,1}(x),\ldots,\chi^0_{s,n_s}(x),\chi^1_{s,1}(x),\ldots,\chi^1_{s,n_s}(x)$ over $A$ as in Proposition \ref{prop:smoothchar} applied to $\theta_s(x,\ybar_s)$, but for $\frac{\epsilon}{2}$. Let $\mathbb{B}$ be the finite sub-algebra of Borel sets in $S_\varphi(A)$ generated by $X_1,\ldots,X_k$ and the formulas $\chi^t_{s,i}(x)$ for $1\leq s\leq k$, $1\leq i\leq n_s$ and $t\in\{0,1\}$. Let $Y_1,\ldots,Y_m$ be the atoms of $\mathbb{B}$, and fix some $p_j\in Y_j$. Choose $a_j\in \cU^x$ realizing $p_j$, and set $Y'_j=(\rho^\varphi_A)^{\text{-}1}(Y_j)$. Define the global Kiesler $\varphi$-measure
\[
\nu=\sum_{j=1}^m \mu(Y'_j)\delta_{a_j}.
\]
Then $\nu|_A$ and $\mu|_A$ agree on all all atoms of $\mathbb{B}$, and hence all sets in $\mathbb{B}$. It is easy to show that, for some $\abar\in \{a_1,\ldots,a_m\}^{<\omega}$, we have $\nu(X)\approx_{\frac{\epsilon}{2}}\Av(\abar)(X)$ for all Borel $X\seq S_\varphi(\cU)$ (in particular, approximate each $\mu(Y'_j)$  by a rational within $\frac{\epsilon}{2m}$). So to establish the primary claim of the proposition, we fix $1\leq s,t\leq k$ and $\bbar\in\cU^{\ybar_s}$, and show that $\mu(\theta_s(x,\bbar)\cap X'_t)\approx_{\frac{\epsilon}{2}}\nu(\theta_s(x,\bbar)\cap X'_t)$. 

By assumption, there is some $1\leq i\leq n$ such that $\psi_{s,i}(\bbar)$ holds. So we have $\chi^0_{s,i}(\cU)\seq\theta_s(\cU,\bbar)\seq\chi^1_{s,i}(\cU)$ and $\mu(\chi^1_{s,i}\backslash \chi^0_{s,i})\leq \frac{\epsilon}{2}$.  Note that
\[
\chi^0_{s,i}(\cU)\cap X'_t\seq \theta_s(\cU,\bbar)\cap X'_t\seq (\chi^0_{s,i}(\cU)\cap X'_t)\cup (\chi^1_{s,i}\backslash\chi^0_{s,i})(\cU),
\]
and so for $\lambda\in\{\mu,\nu\}$, we have
\[
\lambda(\chi^0_{s,i}(x)\cap X'_t)\leq \lambda(\theta_s(x,\bbar)\cap X'_t)\leq \lambda(\chi^0_{s,i}(x)\cap X'_t)+\lambda(\chi^1_{s,i}\backslash\chi^0_{s,i}).
\]
Since $\chi^0_{s,i}\cap X_t$ and $\chi^1_{s,i}\backslash \chi^0_{s,i}$ are both in $\mathbb{B}$, we also have
\[
\nu(\chi^0_{s,i}(x)\cap X'_t)=\nu|_A(\chi^0_{s,i}(x)\cap X_t)=\mu|_A(\chi^0_{s,i}(x)\cap X_t)=\mu(\chi^0_{s,i}(x)\cap X'_t),\text{ and}
\]
\[
\nu(\chi^1_{s,i}\backslash \chi^0_{s,i})=\nu|_A(\chi^1_{s,i}\backslash \chi^0_{s,i})=\mu|_A(\chi^1_{s,i}\backslash \chi^0_{s,i})=\mu(\chi^1_{s,i}\backslash \chi^0_{s,i})\leq \frac{\epsilon}{2}.
\]
So, altogether if $\lambda\in\{\mu,\nu\}$ then $\lambda(\theta_s(x,\bbar)\cap X'_t)$ is between $\mu(\chi^0_{s,i}(x)\cap X'_t)$ and $\mu(\chi^0_{s,i}\cap X'_t)+\frac{\epsilon}{2}$. So $\mu(\theta_s(x,\bbar)\cap X'_t)\approx_{\frac{\epsilon}{2}}\nu(\theta_s(x,\bbar)\cap X'_t)$, as desired.  

For the moreover statement, note that if each $X_t$ is clopen, then so is each $Y_j$, and thus we can assume $p_j$ is realized in $M$,  hence pick $a_j\in M^x$. In the case $k=1$ and $X_1=S_\varphi(M)$, we conclude that $\mu$ is finitely $\theta$-approximated in $M$ for any $\varphi$-generated formula $\theta(x,\ybar)$. 
\end{proof}

The previous result has the following consequences for local measures with smooth extensions (analogous to the global result in \cite[Proposition 7.11]{Sibook}.

\begin{corollary}
Suppose $\mu\in\kM_\varphi(B)$ for some $B\subset\cU$, and assume $\mu$ has a global extension smooth over some $M\prec\cU$ containing $B$. Fix finitely many Borel sets $X_1,\ldots,X_k\seq S_\varphi(B)$ and  $\varphi$-generated formulas $\theta_1(x,\ybar_1),\ldots,\theta_k(x,\ybar_k)$. Then for any $\epsilon>0$ there is some $\bar{p}\in (S(\mu))^{<\omega}$ such that for any $1\leq s,t\leq k$ and $\bbar\in B^{\ybar_s}$,
\[
\mu(\theta_s(x,\bbar)\cap X_t)\approx_\epsilon\Av(\bar{p})(\theta_s(x,\bbar)\cap X_t).
\]
\end{corollary}
\begin{proof}
We may clearly assume $B^y$ is nonempty, since otherwise the statement holds vacuously.
Let $X_{k+1}=S(\mu)$ and $\theta_{k+1}(x,y)$ be $\varphi(x,y)\vee\neg\varphi(x,y)$.   Now fix $\epsilon>0$. Setting $X''_t=(\rho^\varphi_{M,B})^{\text{-}1}(X_t)$, apply the previous proposition to $\mu^*$, $X''_1,\ldots,X''_{k+1}$, and $\theta_1,\ldots,\theta_{k+1}$, with $\frac{\epsilon}{3}$, to obtain $\abar\in (\cU^x)^{<\omega}$. 

Let $X'_t\coloneqq(\rho^\varphi_B)^{\text{-}1}(X_t)=(\rho^\varphi_M)^{\text{-}1}(X''_t)$. Since $\theta_{k+1}(x,b)$ is equivalent to $x=x$ for some/any $b\in B^y$, we have $\Av(\abar)(X'_{t})\approx_{\frac{\epsilon}{3}} \mu^*(X'_t)$
for any $1\leq t\leq k+1$. In particular, 
\[
\Av(\abar)(X'_{k+1})\approx_{\frac{\epsilon}{3}} \mu^*(X'_{k+1})=\mu^*|_B(X_{k+1})=\mu(S(\mu))=1.
\]
So if we enumerate $\abar=(a_1,\ldots,a_m)$, then (without loss of generality) $\tp_\varphi(a_j/\cU)\in X'_{k+1}$ for all 
$1\leq j\leq m'$, where $m'\geq (1-\frac{\epsilon}{3})m$. So $p_j\coloneqq\tp_\varphi(a_j/B)\in S(\mu)$ for all $1\leq j\leq m'$.  We show that the desired conclusion holds for $\bar{p}\coloneqq(p_1,\ldots,p_{m'})$. 

Fix $1\leq s,t\leq k$ and $\bar{b}\in B^{\ybar_s}$. Then
\begin{multline*}
\mu(\theta_s(x,\bbar)\cap X_t)=\mu^*|_B(\theta_s(x,\bbar)\cap X_t)=\mu^*(\theta_s(x,\bbar)\cap X'_t)\\
\approx_{\frac{\epsilon}{3}} \Av(a_1,\ldots,a_m)(\theta_s(x,\bbar)\cap X'_t)=\Av(p_1,\ldots,p_m)(\theta_s(x,\bbar)\cap X_t).
\end{multline*}
Set $\alpha=\Av(p_1,\ldots,p_m)(\theta_s(x,\bbar)\cap X_t)$ and $\beta=\Av(\bar{p})(\theta_s(x,\bbar)\cap X_t)$. Then
\begin{multline*}
|\alpha-\beta|=\left|\frac{m'-m}{m}\beta+\frac{1}{m}\sum_{j=m'+1}^mp_j(\theta_s(x,\bbar)\cap X_t)\right|\leq 2\frac{m-m'}{m}\leq \frac{2\epsilon}{3}.\\
\end{multline*}
So altogether we have $\mu(\theta_s(x,\bbar)\cap X_t)\approx_{\frac{\epsilon}{3}} \alpha\approx_{\frac{2\epsilon}{3}}\beta$, as desired.
\end{proof}

As we will see later, the assumptions of the previous corollary hold for any local measure with respect to an NIP formula $\varphi(x,y)$.

\subsection{Packing numbers}\label{sec:packing}

We again fix an $\cL$-formula $\varphi(x,y)$. In this section, we focus on  $\varphi$-measures over an arbitrary model $M\preceq\cU$. The general question we are interested in is when a measure in $\kM_\varphi(M)$ has a global extension satisfying  nice properties (e.g., those defined in the previous subsection). We will see that large ``separated families" (defined below) present an obstacle to the existence of such an extension. To get started, we recall the following notions from discrete geometry.

\begin{definition}
Let $(X,\mathcal{B},\mu)$ be a finitely additive probability space. 
\begin{enumerate}[$(1)$]
\item A subset $\mathcal{F}\seq\mathcal{B}$ is \textbf{$\epsilon$-separated} if $\mu(A\ssd B)\geq\epsilon$  for all distinct $A,B\in\mathcal{F}$. 
\item Given $\mathcal{S}\seq\mathcal{B}$, the \textbf{$\epsilon$-packing number of $\mathcal{S}$}  is
\[
\pn(\cS,\mu,\epsilon)=\sup\{|\cF|:\cF\seq\cS\text{ is finite and $\epsilon$-separated}\}.
\]
\end{enumerate}
\end{definition}

We apply this definition in our model theoretic setting in the usual way.

\begin{definition}
Given $M\preceq\cU$ and $\mu\in\kM_\varphi(M)$, let $\pn(\varphi,\mu,\epsilon)$ denote $\pn(\cS_{\varphi,M},\mu,\epsilon)$, where $S_{\varphi,M}$ is the collection of instances of $\varphi(x,y)$ over $M$. 
\end{definition}

\begin{remark}\label{rem:mono}
If $\mu\in\kM_\varphi(M)$ and $\nu\in\kM_\varphi(\cU)$ is a global extension of $\mu$ then $\pn(\mu,\varphi,\epsilon)\leq \pn(\nu,\varphi,\epsilon)$ for any $\epsilon>0$.
\end{remark}

Define the $\varphi$-generated formula
\[
\varphi_{\subt}(x;y_1,y_1):= \varphi(x,y_1)\fsd \varphi(x,y_2).
\] 
Then, given $M\preceq \cU$ and $\mu\in\kM_\varphi(M)$, we have $\pn(\mu,\varphi,\epsilon)\geq n$ if and only if there are $b_1,\ldots,b_n\in M^y$ such that $\mu(\varphi_{\subt}(x;b_i,b_j))\geq\epsilon$ for all $1\leq i<j\leq n$.
The next result follows from  the proof of \cite[Proposition 2.18]{ChStNIP}.

\begin{fact}[Chernikov \& Starchenko]\label{fact:CS}
If $\mu\in\kM_{\varphi}(M)$ is finitely $\varphi_{\subt}$-approximated, then $\pn(\varphi,\mu,\epsilon)<\infty$ for all $\epsilon>0$.
\end{fact}

Via Remark \ref{rem:mono}, one then concludes:

\begin{corollary}\label{cor:CS}
If $\mu\in\kM_{\varphi}(M)$ has a global finitely $\varphi_{\subt}$-approximated extension, then $\pn(\varphi,\mu,\epsilon)<\infty$ for all $\epsilon>0$.
\end{corollary}

The previous corollary is stronger than Fact \ref{fact:CS} since if $\mu\in\kM_\varphi(M)$ is finitely $\varphi_{\subt}$-approximated then it has a global extension that is finitely $\varphi_{\subt}$-approximated in $M$. On the other hand, there are measures over models that are not finitely approximated, but do have global  finitely approximated extensions. 

Recall that if $\varphi(x,y)$ is an $\cL$-formula and $\mu\in\kM_\varphi(\cU)$ is finitely $\varphi_{\subt}$-approximated in $A\subset\cU$, then it is strongly $\varphi_{\subt}$-definable over $A$ by Proposition \ref{prop:imps}$(d)$. Therefore the next result generalizes Corollary \ref{cor:CS}.

\begin{proposition}\label{prop:phipack}
If $\mu\in\kM_\varphi(M)$ has a global strongly $\varphi_{\subt}$-definable extension, then $\pn(\varphi,\mu,\epsilon)<\infty$ for all $\epsilon>0$.
\end{proposition}
\begin{proof}
In light of Remark \ref{rem:mono}, it suffices to assume that $M=\cU$ and $\mu$ is strongly $\varphi_{\subt}$-definable over some small set $A\subset\cU$.
Fix $\epsilon>0$. Toward a contradiction, suppose $\#(\mu,\varphi,\epsilon)=\infty$. Since $F^{\varphi_{\ssubt},s}_{\mu,A}$ is a continuous function on the Stone space $S_{\varphi^*_{\ssubt}}(A)$, it can be uniformly approximated by a finite linear combination of characteristic functions of clopen sets (see, e.g., \cite[Fact 2.10]{GannNIP}). So there are finitely many real numbers $r_1,\ldots,r_n\in[0,1]$ and $\varphi_{\subt}^*$-formulas $\psi_1(y_1,y_2),\ldots,\psi_n(y_1,y_2)$  such that, for any $b,c\in\cU^y$, 
\[
\mu(\varphi_{\subt}(x;b,c))\approx_{\frac{\epsilon}{2}}\sum_{i=1}^n r_i\boldsymbol{1}_{\psi_i}(b,c).
\]

Let $A_0$ be the collection of parameters that occur in $\psi_1,\ldots,\psi_n$. Note that $S_{\varphi^*}(A_0)$ is finite (since $A_0$ is finite). Since $\pn(\mu,\varphi,\epsilon)=\infty$, we can choose $b,c\in\cU^y$ such that $\mu(\varphi_{\subt}(x;b,c))\geq\epsilon$ and $\tp_{\varphi^*}(b/A_0)=\tp_{\varphi^*}(c/A_0)$. The first condition implies $\sum_{i=1}^n r_i\boldsymbol{1}_{\psi_i}(b,c)>\frac{\epsilon}{2}$, while the second condition implies that  $\cU\models\neg\varphi_{\subt}(a;b,c)$ for any $a\in A_0$, and so $\tp_{\varphi^*_{\ssubt}}(b,c/A_0)=\tp_{\varphi^*_{\ssubt}}(b,b/A_0)$. In particular, for all $1\leq i\leq n$, we have $\boldsymbol{1}_{\psi_i}(b,c)=\boldsymbol{1}_{\psi_i}(b,b)$ since $\psi_i$ is a $\varphi^*_{\subt}$-formula over $A_0$. Altogether, 
\[
\frac{\epsilon}{2}<\sum_{i=1}^nr_i\boldsymbol{1}_{\psi_i}(b,c)=\sum_{i=1}^nr_i\boldsymbol{1}_{\psi_i}(b,b)\approx_{\frac{\epsilon}{2}} \mu(\varphi_{\subt}(x;b,b))=0,
\]
which is a contradiction.
\end{proof}

In comparing Corollary \ref{cor:CS} and Proposition \ref{prop:phipack}, we note that strong definability is a much weaker assumption than finite approximability. For example, in DLO, if $\varphi(x,y)$ is $x<y$ then the type $p\in S_\varphi(\cU)$ containing $x>b$ for all $b\in\cU$ is strongly $\theta$-definable for any $\varphi$-generated formula $\theta(x,\ybar)$, but is not  even finitely satisfiable (with respect to  $\varphi(x,y)$) in any small model.

We end this section with some remarks about non-local measures. We say that a  measure $\mu\in\kM_x(\cU)$ is \emph{(strongly) definable} over some $A\subset\cU$ if, for any formula $\varphi(x,y)$ and any $\varphi$-generated formula $\theta(x,\ybar)$, $\mu|_\varphi$ is (strongly) $\theta$-definable over $A$. Recall that any strongly definable measure is definable (Proposition \ref{prop:imps}$(b)$). Since the notion of strong definability is not prevalent in the literature, we also point out that it is a relatively weak property in general (when compared to, say, finite approximation or smoothness). For example, any \emph{dfs} measure is strongly definable by the following result (which is largely evident from \cite[Proposition 4.5]{GannNIP} or \cite[Proposition 2.9]{CoGa}).

\begin{proposition}\label{prop:dfssd}
Let $\theta(x,\ybar)$ be a $\varphi$-generated formula, and suppose $\mu\in\kM_\varphi(\cU)$ is finitely $\theta_{\subt}$-satisfiable in some $A\subset\cU$ (i.e., every instance of $\theta_{\subt}$ with positive measure has a solution in $A^x$). Then $\mu$ is strongly $\theta$-invariant over $A$. Thus if $\mu$ is also $\theta$-definable, then $\mu$ is strongly $\theta$-definable over $A$. 
\end{proposition}
\begin{proof}
First we show that $\mu$ is strongly $\theta$-invariant. Fix $b,b'\in\cU^y$ such that $\tp_{\theta^*}(b/A)=\tp_{\theta^*}(b'/A)$. If $\mu(\varphi(x,b))\neq\mu(\varphi(x,b'))$ then $\mu(\theta_{\subt}(x;b,b'))>0$, which  which contradicts finite $\theta_{\subt}$-satisfiability of $\mu$ in $A$ and the choice of $b,b'$. This proves the first claim. The second claim follows from Proposition \ref{prop:imps}$(c)$. 
\end{proof}

Since \emph{dfs} measures are widely studied and more naturally occurring, we also make note of the following corollary.

\begin{corollary}
If $\mu\in\kM_x(M)$ has a global dfs extension, then $\pn(\varphi,\mu,\epsilon)<\infty$ for any $\cL$-formula $\varphi(x,y)$ and any $\epsilon>0$. 
\end{corollary}

Indeed, the above conclusion holds when $\mu$ has a global strongly definable extension.

\subsection{Smooth extensions for NIP formulas}\label{sec:smoothext}

In this section, we prove the local characterization of NIP formulas in terms of the existence of smooth extensions for measures over small models. This characterization will pass through the fact that an NIP formula admits finite packing numbers with respect to any measure. This fact is well known in the model theory literature, though usually in the context of NIP theories (see e.g., \cite[Proposition 3.3]{HPP} and \cite[Theorem 3.14]{Keis}). The standard model-theoretic proof involves expanding the language to artificially make the measure in question definable, and then applying compactness in order to access alternation numbers of indiscernible sequences with respect to NIP formulas. Alternatively, one can  directly apply the following important fact about VC-dimension.

\begin{fact}[Haussler's Packing Lemma \cite{HaussPL}]
Suppose $(X,\cB,\mu)$ is a probability space, and $\cS\seq\cB$ is a set system with VC-dimension $d$. Then $\pn(\mu,\cS,\epsilon)\leq (4e^2/\epsilon)^d$ for all $\epsilon>0$.
\end{fact}

 Note that one can deduce Haussler's Packing Lemma for a finitely additive probability space by applying the previous statement to the corresponding Stone space. Therefore we immediately obtain the following quantitative version of \cite[Proposition 3.3]{HPP}. Given $d\geq 1$, we say that an $\cL$-formula $\varphi(x,y)$ is \textbf{$k$-NIP} (with respect to $T$) if there do not exist sequences $(a_i)_{i=1}^k$ from $\cU^x$ and $(b_I)_{I\seq[k]}$ from $\cU^y$ such that $\varphi(a_i,b_I)$ holds if and only if $i\in I$. In other words, $\varphi(x,y)$ is $k$-NIP if and only if $\cS_{\varphi,\cU}$ has VC-dimension at most $k-1$. 

\begin{corollary}\label{cor:Haussler}
Suppose $\varphi(x,y)$ is  $k$-NIP. Then $\pn(\mu,\varphi,\epsilon)\leq (4e^2/\epsilon)^{k-1}$ for all $\mu\in\kM_\varphi(\cU)$ and $\epsilon>0$.
\end{corollary}

\begin{remark}
Since the proof of Haussler's Packing Lemma is fairly complicated, we note that the same statement with a slightly weaker bound can be deduced quickly from the Sauer-Shelah Lemma, which is more familiar and elementary (see, e.g., \cite[Lemma 6.4]{Sibook}). This argument is well-known in the combinatorics literature, and  is short enough to include here (following the remarks after \cite[Theorem 2.1]{Moran-Yeh}).

Let $(X,\cB,\mu)$ be a probability space, and suppose $\cS\seq\cB$ has VC-dimension $d$. Let $\cF\seq \cS$ be a finite $\epsilon$-separated set of size $k$, and let $n$ be the least integer such that $n\geq d$ and ${k\choose 2}(1-\epsilon)^n<1$.  Consider the product measure $\mu^n$ on $X^n$, and let 
\[
E=\bigcup_{\{A,B\}\in{\cF\choose 2}} (X\backslash(A\ssd B))^n.
\]
Then $\mu^n(E)\leq {k\choose 2}(1-\epsilon)^n<1$. So we can fix some $\xbar\in X^n\backslash E$. Then for all distinct $A,B\in\cF$ there is some $x_i\in A\ssd B$, and so $A\cap \{x_1,\ldots,x_n\}\neq B\cap \{x_1,\ldots,x_n\}$. So the map $A\mapsto A\cap\{x_1,\ldots,x_n\}$ is injective on $\cF$. By the Sauer-Shelah Lemma,
\[
k=|\cF|\leq \pi_{\cS}(n)\leq (en/d)^d,
\]
where $\pi_{\cS}$ denotes the shatter function for $\cS$.
Together with the choice of $n$, one can use this to compute an explicit bound of the form $k\leq O_d(1/\epsilon)^{(1+o_d(1))d}$.
\end{remark}

Next, we note that with the above tools in hand, the local analogue of one direction of Fact \ref{fact:Keisler} becomes relatively straightforward.

\begin{lemma}\label{lem:NIPsmooth}
Assume $\varphi(x,y)$ is NIP, and suppose $\mu\in\kM_\varphi(M)$ where $M\prec\cU$. Then there is a global extension $\hat{\mu}\in\kM_\varphi(\cU)$ of $\mu$ which is smooth over some small model of size at most $|M|$.  
\end{lemma}
\begin{proof}
We can follow the same strategy as the proof in the global case of an NIP theory (see \cite[Proposition 7.9]{Sibook}). In particular, if the result fails then one obtains a $\varphi$-generated formula $\theta(x,\ybar)$, a measure $\mu^*\in\kM_\varphi(\cU)$, and some $\epsilon>0$, such that $\pn(\theta,\mu^*,\epsilon)=\infty$. Since $\theta(x,\ybar)$ is still NIP, this contradicts Corollary \ref{cor:Haussler}. 
\end{proof}

Now we prove the main local result.

\begin{thm}\label{thm:localNIPchar}
The following are equivalent.
\begin{enumerate}[$(i)$]
\item $\varphi(x,y)$ is NIP.
\item For any $M\models T$, every $\mu\in\kM_\varphi(M)$ has a smooth global extension.
\item For any $M\models T$ of size at most $|T|$, every $\mu\in\kM_\varphi(M)$ has a strongly $\varphi_{\subt}$-definable global extension.
\item $\pn(\mu,\varphi,\epsilon)<\infty$ for all $\mu\in\kM_\varphi(\cU)$ and $\epsilon>0$. 
\end{enumerate}
\end{thm}
\begin{proof}
 $(i)\Rightarrow(ii)$. This is Lemma \ref{lem:NIPsmooth}.
 
 $(ii)\Rightarrow(iii)$. This follows from Propositions \ref{prop:imps}$(d)$ and \ref{prop:smoothfam}.

$(iii)\Rightarrow (iv)$. If $(iii)$ holds then, by Proposition \ref{prop:phipack}, we have $\pn(\mu,\varphi,\epsilon)<\infty$ for all $\mu\in\kM_\varphi(M)$ and $\epsilon>0$, where $|M|\leq |T|$. This suffices to conclude $(iv)$ since if $\mu\in\kM_\varphi(\cU)$ and $\pn(\mu,\varphi,\epsilon)=\infty$, then there is some $M\prec\cU$ of size at most $|T|$ such that $\pn(\mu|_M,\varphi,\epsilon)=\infty$.

$(iv)\Rightarrow (i)$. Suppose $\varphi(x,y)$ has the independence property. Then we can find an infinite set $B\seq \cU^y$ such that if $X,Y\subset B$ are finite and disjoint, then there is some $a\in \cU^x$ such that $\varphi(x,b)$ holds for all $b\in X$ and $\neg\varphi(a,b)$ holds for all $b\in Y$. In particular, for any $b_1,\ldots,b_m,c_1,\ldots,c_n\in B$, if $\bigcap_{i=1}^m\varphi(\cU^x,b_i)\seq\bigcup_{j=1}^n\varphi(\cU^x,c_j)$, then $b_i=c_j$ for some $i$ and $j$. This independence among $B$-instances of $\varphi(x,y)$ implies the existence of a ``coin-flipping" measure analogous to the usual probabilistic measure on the random graph (see Lemma A.6 and Corollary A.7 of \cite{CoGaHa} for details). In particular, there is a measure $\mu\in \kM_\varphi(\cU)$ such that if $b_1,\ldots,b_n\in B$ are pairwise distinct and $\epsilon_1,\ldots,\epsilon_n\in\{0,1\}$ then
\[
\mu\left(\bigwedge_{i=1}^n\varphi^{\epsilon_i}(x,b_i)\right)=\frac{1}{2^n}.
\]
Thus if $b,b'\in B$ are distinct then 
\[
\mu(\varphi_{\subt}(x,b,b'))=\mu(\varphi(x,b)\wedge\neg\varphi(x,b'))+\mu(\neg\varphi(x,b)\wedge\varphi(x,b'))=\frac{1}{4}+\frac{1}{4}=\frac{1}{2}.
\]
Therefore $\pn(\mu,\varphi,1/2)=\infty$.
\end{proof}

We can also use this to obtain an analogous characterization of NIP theories, which refines Fact \ref{fact:Keisler}. Before stating it, recall from Remark \ref{rem:nosmooth} that if $\mu\in\kM_x(\cU)$ is smooth then it is not necessarily the case that $\mu|_{\varphi}\in\kM_\varphi(\cU)$ is smooth for a particular $\cL$-formula $\varphi(x,y)$. (This is in contrast to other notions such as finite approximability and definability.) However, if $\mu\in\kM_x(\cU)$ is smooth over $M\prec\cU$ then it is finitely approximated in $M$, and thus $\mu|_\varphi$ is finitely approximated in $M$ for any $\cL$-formula $\varphi(x,y)$.

\begin{thm}
The following are equivalent.
\begin{enumerate}[$(i)$]
\item $T$ is NIP.
\item For any $M\models T$, every $\mu\in\kM_x(\cU)$ has a smooth global extension.
\item For any $M\models T$ of size at most $|T|$ and any singleton variable $x$, every $\mu\in\kM_x(\cU)$ has a strongly definable global extension.
\item $\pn(\mu,\varphi,\epsilon)<\infty$ for any $\cL$-formula $\varphi(x,y)$, $\mu\in\kM_x(\cU)$, and $\epsilon>0$.
\end{enumerate}
\end{thm}
\begin{proof}
$(i)\Leftrightarrow (v)$. This follows from the corresponding equivalence in Theorem \ref{thm:localNIPchar}. 

$(i)\Rightarrow (ii)$. This is the well-known non-local analogue of Lemma \ref{lem:NIPsmooth}.  

$(ii)\Rightarrow (iii)$. This follows from the fact that if $\mu\in\kM_x(\cU)$ is smooth then it is finitely approximated (together with the above remarks). 

$(iii)\Rightarrow (i)$. If $(iii)$ holds then, as in the proof of  Theorem \ref{thm:localNIPchar}, we see that any $\cL$-formula $\varphi(x,y)$ with $|x|=1$ is NIP. This suffices to conclude that $T$ is NIP.
\end{proof}

\end{document}